\documentclass[10pt]{article}%

\usepackage[utf8]{inputenc}

\usepackage{amsmath}

\usepackage{amsfonts}
\usepackage{mathrsfs}
\usepackage{amssymb, color}
\usepackage[linkcolor=black,anchorcolor=black,citecolor=black]{hyperref}
\usepackage{graphicx}
\numberwithin{equation}{section}
\usepackage[body={15.5cm,21cm}, top=3cm]{geometry}%
\providecommand{\U}[1]{\protect\rule{.1in}{.1in}}
\providecommand{\U}[1]{\protect \rule{.1in}{.1in}}
\newtheorem{theorem}{Theorem}[section]

\newtheorem{lemma}[theorem]{Lemma}

\newtheorem{proposition}[theorem]{Proposition}
\newtheorem{remark}[theorem]{Remark}

\newenvironment{proof}[1][Proof]{\noindent \textbf{#1.} }{\  \rule{0.5em}{0.5em}}
\DeclareMathOperator*{\esssup}{ess\,sup}

\def \E{\mathsf{E}}

\def \P{\mathsf{P}}

\usepackage [numbers,sort&compress] {natbib}

\begin{document}
	\title{Propagation of Chaos for Mean-field Mean Reflected Backward Stochastic Differential Equations}
	\author{ 	Hanwu Li\thanks{Research Center for Mathematics and Interdisciplinary Sciences, Shandong University, Qingdao 266237, Shandong, China. lihanwu@sdu.edu.cn.}
	\thanks{Frontiers Science Center for Nonlinear Expectations (Ministry of Education), Shandong University, Qingdao 266237, Shandong, China.}
    \thanks{Shandong Province Key Laboratory of Financial Risk, Shandong University, Qingdao 266237, Shandong, China.}}
	\date{}
	\maketitle
	
	\begin{abstract}
       In this paper, we establish a propagation of chaos result for mean-field mean reflected backward stochastic differential equations (BSDEs), where both the generator and  constraint depend on the distribution of the solution. When the generator does not rely on $z$, under mild Lipschitz and integrability conditions, we prove existence and uniqueness of the solution to the interacting particle system for general reflections. We are able to consider the case where the generator depends on $z$ when the reflection is linear. In both cases, we obtain the convergence rate of solution to the interacting particle system towards the solution to the mean-field mean reflected BSDEs. 
	\end{abstract}

    \textbf{Key words}: backward stochastic differential equations, mean field, mean reflection, interacting particle system, propagation of chaos

    \textbf{MSC-classification}: 60H10
	
\section{Introduction}


The concept of reflected backward stochastic differential equations (reflected BSDEs) was first proposed by El Karoui et al. \cite{EKPPQ}, originally motivated by constrained problems in mathematical finance, particularly the pricing of American options. Specifically, the constraint is characterized by 
$Y_t\geq S_t$, where $S$ is a predetermined process referred to as the obstacle. Owing to its extensive applications in mathematical finance, economics and partial differential equations, the theoretical research on reflected BSDEs has witnessed rapid development. A variety of generalized frameworks have been established over the years, covering doubly reflected BSDEs, multidimensional reflected BSDEs, BSDEs with non-Lipschitz generators, systems beyond the Brownian motion framework, and obstacles with low regularity. Relevant research can be found in \cite{CM,CK,GP,HT,GIOOQ,K1,KLQT,PX,PR} and the references listed therein.

It is worth noting that the constraints in all the above mentioned papers are made on the paths of the solution. Motivated by the superhedging of contingent claims under running risk management constraint, Briand, Elie and Hu \cite{BEH} introduced the mean reflected BSDE, where the constraint is given in terms of the distribution of the solution written as follows:
\begin{align}\label{constraint1}
\E[l(t,Y_t)]\geq 0, \ t\in[0,T].
\end{align}
Here, $l$ is a given loss function. A typical example is $l(t,x)=I_{\{y\geq u_t\}}-v_t$ and therefore the constraint amounts to say that the process $Y$ is required to beat the benchmark $u$ with a probability greater than $v_t$ at any time $t\in[0,T]$.  Recently, Djehiche, Elie and Hamad\`{e}ne \cite{DEH} studied a new class of reflected BSDEs of mean-field type, where both the generator and  constraint depend on the distribution of the $Y$-component. More precisely, the constraint is written as 
\begin{align}\label{constraint2}
    Y_t\geq l(t,Y_t,\P_{Y_t}), \ t\in[0,T].
\end{align}
Mean-field reflected BSDEs  admit important applications in the pricing of life insurance contracts with surrender options. Since then, the mean reflected BSDE and mean-field reflected BSDE have attracted considerable attention. Chen, Hamad\`{e}ne and Mu \cite{CHM}, Falkowski and S\l omi\'{n}ski \cite{FS} and Li \cite{Li24} investigated the case of double reflections, respectively.  The well-posedness for such equations with non-Lipschitz generators was established in \cite{CZ,HHLLW,HMW',HMW,LS}. Qu and Wang \cite{QW} and Niu, Qu and Wang \cite{NQW} further investigated multi-dimensional mean reflected BSDEs. In addition, extensions to mean-field reflected BSDEs with jumps have been presented in \cite{DD,DDZ,LX1,LX2}.


Since the involved constraints rely on the distribution of solutions, mean reflected BSDEs share a close connection with mean-field BSDEs. For instance, analogous to the penalization approach adopted for classical reflected BSDEs, solutions to mean reflected BSDEs can be approximated by those of a class of penalized mean-field BSDEs (see \cite{BEH,CHM,Li24,LS}). On the other hand, the research on mean-field BSDEs originates from exploring the limiting behavior of high-dimensional forward-backward SDE systems, which characterize the dynamics of large-scale particle populations (see \cite{BDLP}). Accordingly, a natural question arises: whether mean reflected BSDEs can be regarded as the asymptotic dynamics of such particle systems. The answer is affirmative. Specifically, existing literature in \cite{BH,DD,DDZ,LX2,LN} has established the propagation of chaos for mean reflected BSDEs and mean-field reflected BSDEs in respective settings.


This paper aims to establish a propagation of chaos result for mean-field mean reflected BSDEs originally introduced in \cite{HMW}. Inspired by the approximation schemes proposed in \cite{BH} for mean reflected BSDEs and in \cite{BDLP} for standard mean-field BSDEs, we construct a candidate interacting particle system formulated as a multi-dimensional reflected BSDE. We first prove the existence and uniqueness of solutions for this multi-dimensional reflected BSDE. In the linear loss function setting, the well-posedness of the interacting particle system follows directly from the conclusions in \cite{GP}. For the nonlinear case, we first establish solvability for constant generators independent of  $Y$ and $Z$
 via the Snell envelope method. We then adopt a contraction mapping argument to derive the existence and uniqueness of solutions when the generator depends merely on the  $Y$-component. The approximation based on the interacting particle system crucially requires uniform-in-time control of the Wasserstein distance between the empirical measure of an i.i.d. process sequence and its marginal law. Under sufficient smoothness assumptions on the loss function, we leverage the property that the empirical measure is composed of i.i.d. copies of diffusion processes.


Although the solution to mean-field mean reflected BSDEs has been constructed in \cite{HMW}, it remains numerically intractable. Given the existence of multiple algorithms for particle systems based on empirical distributions, our results provide a theoretical basis for the numerical computation of solutions to mean-field mean reflected BSDEs. In comparison with \cite{BH,LN}, the mean reflected BSDE considered in this paper is of mean-field type, meaning that its generator also depends on the expectation of the solution. On the other hand, despite the seemingly more complex constraint in \cite{DD,DDZ} (see \eqref{constraint2}), it cannot degenerate to the constraint \eqref{constraint1} employed herein. Additionally, the loss function  $l$ in \cite{DD,DDZ} imposes additional assumptions on its Lipschitz constant. Furthermore, the distribution of the $Z$-component does not enter the generator of mean-field reflected BSDEs, even when the loss function $l$ in \eqref{constraint2} is linear.

The paper is organized as follows. We introduce some notations and results for mean-field mean reflected BSDEs in Section 2. In Section 3, we first prove the existence and uniqueness result for the solution to interacting particle system when the loss function is nonlinear and the generator is independent of the $Z$-component, and establish the convergence rate of the propagation of chaos. Finally, in the last section, we consider the case of linear reflection, where the generator may depend on both $Y$ and $Z$ as well as their expectations.

 \section{Preliminaries} 
 
 In this section, we first recall some basic results about mean-field mean reflected BSDEs. Given a fixed time $T$, let $(\Omega,\mathcal{F},\P)$ be a complete probability space under which $B=\{B_t\}_{t\in[0,T]}$ is a standard real valued Brownian motion. The augmented natural filtration generated by $B$ is denoted by $\mathbb{F}=\{\mathcal{F}_t\}_{t\in[0,T]}$. $\mathcal{P}$ is the sigma algebra of all $\mathbb{F}$-progressive sets of $\Omega\times[0,T]$. 
 For any given positive integer $m$ and any given filtration $\mathbb{G}=\{\mathcal{G}_t\}_{t\in[0,T]}$, we first introduce the following notations, which will be frequently used in this paper.

\begin{itemize}
\item $L^2(\mathcal{G}_t;\mathbb{R}^m)$: the set of $\mathcal{G}_t$-measurable random variables $\xi$ taking values in $\mathbb{R}^m$ such that $\E[|\xi|^2]<\infty$, $t\in[0,T]$;
\item $\mathcal{S}^2(\mathbb{G};\mathbb{R}^m)$: the set of  $\mathbb{G}$-adapted continuous processes $Y$ on $[0,T]$ taking values in $\mathbb{R}^m$ with $\E[\sup_{t\in[0,T]}|Y_t|^2]<\infty$;
\item $\mathcal{H}^2(\mathbb{G};\mathbb{R}^m)$: the set of $\mathbb{G}$-progressively measurable processes $Z$ taking values in $\mathbb{R}^m$ such that $\E[\int_0^T|Z_t|^2dt]<\infty$;
\item $\mathcal{A}^2(\mathbb{G})$: the set of all continuous, $\mathbb{G}$-adapted and non-decreasing processes $K$ such that $K_0=0$ and $\E[|K_T|^2]<\infty$. 
\item $C[0,T]$: the set of continuous functions from $[0,T]$ to $\mathbb{R}$;
\item $I[0,T]$: the set of functions in $C[0,T]$ starting from the origin which is nondecreasing;
\item $C^{1,2}_b([0,T]\times\mathbb{R})$: the space of all continuous functions on $[0,T]\times \mathbb{R}$, which are continuously differentiable in their first
variable and twice continuously differentiable in their second variable, and all
derivatives are bounded. 
\end{itemize}
For the case that $m=1$ and $\mathbb{G}=\mathbb{F}$, we always omit them in the brackets.

 
Consider the following type of mean-field mean reflected BSDE (see \cite{HMW}):
\begin{equation}\label{nonlinearyz}
\begin{cases}
Y_t=\xi+\int_t^T f(s,Y_s,\E[Y_s],Z_s,\E[Z_s])ds-\int_t^T Z_s dB_s+K_T-K_t,\ t\in[0,T], \\
\E[h(t,Y_t)]\geq 0, \ t\in[0,T] \textrm{ and } \int_0^T \E[h(t,Y_t)]dK_t=0,
\end{cases}
\end{equation}
where the generator $f:\Omega\times[0,T]\times \mathbb{R}^4\rightarrow \mathbb{R}$ and the loss function $h:\Omega\times[0,T]\times\mathbb{R}\rightarrow\mathbb{R}$ are measurable maps with respect to $\mathcal{P}\times \mathcal{B}(\mathbb{R}^4)$ and $\mathcal{F}_T\times\mathcal{B}([0,T])\times\mathcal{B}(\mathbb{R})$, respectively. We propose the following conditions on the terminal value $\xi$, the generator $f$ and the loss function $h$.

\begin{itemize}
 \item[(H1)] $\xi\in L^2(\mathcal{F}_T)$ and $\E[h(T,\xi)]\geq 0$.
 \item[(H2)] The process $\{f(t,0,0,0,0)\}_{t\in[0,T]}$ belongs to $\mathcal{H}^2$ and there exists a constant $L>0$ such that for any $t\in[0,T]$, $y_i,z_i,y'_i,z'_i\in\mathbb{R}$, $i=1,2$,
 \begin{align*}
     |f(t,y_1,y'_1,z_1,z'_1)-f(t,y_2,y'_2,z_2,z'_2)|\leq L(|y_1-y_2|+|z_1-z_2|+|y'_1-y'_2|+|z'_1-z'_2|).
 \end{align*}
 \item[(H3)] There exists a constant $L>0$, such that 
 \begin{itemize}
     \item[1.] $(t,y)\mapsto h(t,y)$ is continuous,
     \item[2.] for any $t\in[0,T]$, $y\mapsto h(t,y)$ is strictly increasing,
     \item[3.] for any $t\in[0,T]$, $\lim_{y\rightarrow \infty}\E[h(t,y)]>0$,
     \item[4.] for any $(t,y)\in[0,T]\times \mathbb{R}$, $|h(t,y)|\leq L(1+|y|)$.
 \end{itemize}
\item[(H4)]  $h$ is bi-Lipschitz continuous in $y$, i.e., there exists $0<m\leq M$, such that for any $t\in[0,T]$, $y_1,y_2\in\mathbb{R}$,
\begin{align*}
m|y_1-y_2|\leq |h(t,y_1)-h(t,y_2)|\leq M|y_1-y_2|.
\end{align*}
 \end{itemize} 
 In order to construct the solution to \eqref{nonlinearyz}, for each $t\in[0,T]$, the following operator $L_t:L^2(\mathcal{F}_T)\rightarrow\mathbb{R}$ plays an important role:
 \begin{align*}
     L_t(\eta):=\inf\{x\geq 0:\E[h(t,x+\eta)]\geq 0\}.
 \end{align*}
 
 \begin{theorem}[\cite{BEH,HMW}]\label{thm2.2}
 Under Assumptions (H1)-(H4), the mean-field mean reflected BSDE associated with $(\xi,f,h)$ admits a unique solution $(Y,Z,K)\in \mathcal{S}^2\times\mathcal{H}^2\times I[0,T]$. Moreover, for each $t\in[0,T]$, we have 
 \begin{align*}
     K_T-K_t=\sup_{s\in[t,T]}L_s(\bar{Y}_s), 
 \end{align*}
 where
 \begin{align*}
\bar{Y}_s:=\E\left[\xi+\int_s^T f(r,Y_r,\E[Y_r],Z_r,\E[Z_r])ds\Big|\mathcal{F}_s\right].
\end{align*}
 \end{theorem}

\begin{remark}\label{rem MFMRBSDE}
    (1) The generator of mean-field mean reflected BSDE may depends on the distribution of $Y$ and $Z$. That is, the dynamics of \eqref{nonlinearyz} can be written as follows
    \begin{align*}
        Y_t=\xi+\int_t^T f(s,Y_s,\P_{Y_s},Z_s,\P_{Z_s})ds-\int_t^T Z_s dB_s+K_T-K_t.
    \end{align*}
    For $p\geq 1$, let $\mathcal{P}_p(\mathbb{R})$ be the collection of all probability measures on $(\mathbb{R},\mathcal{B}(\mathbb{R}))$ with finite $p$th-moment, endowed with the $p$-Wasserstein distance $W_p$. 
     Suppose that $f$ satisfies (H2'), where
     \begin{itemize}
         \item[(H2')]  The process $\{f(t,0,\delta_0,0,\delta_0)\}_{t\in[0,T]}$ belongs to $\mathcal{H}^2$ and there exists a constant $L>0$ such that for any $t\in[0,T]$, $y_i,z_i\in\mathbb{R}$ and $\mu_i,\nu_i\in\mathcal{P}_1(\mathbb{R})$, $i=1,2$,
 \begin{align*}
     |f(t,y_1,\mu_1,z_1,\nu_1)-f(t,y_2,\mu_2,z_2,\nu_2)|\leq L(|y_1-y_2|+|z_1-z_2|+W_1(\mu_1,\mu_2)+W_1(\nu_1,\nu_2)),
 \end{align*}
     \end{itemize}
     Then, the mean-field mean reflected BSDE with parameters $(\xi,f,h)$ has a unique solution $(Y,Z,K)$. 

(2) When the generator does not depends on the distribution of $Z$ and has quadratic growth in $Z$, the mean-field mean reflected BSDE admits a unique solution (see Theorem 2 and Theorem 3 in \cite{HMW}).

(3) In \cite{DEH}, the authors introduce another type of mean-field reflected BSDEs
\begin{equation}\label{nonlinearyz'}
    \begin{cases}
Y_t=\xi+\int_t^T f(s,Y_s,\P_{Y_s},Z_s)ds-\int_t^T Z_s dB_s+K_T-K_t,\ t\in[0,T], \\
Y_t\geq l(t,Y_t,\P_{Y_t}), \ t\in[0,T] \textrm{ and } \int_0^T (Y_t-l(t,Y_t,\P_{Y_t})dK_t=0.
\end{cases}
\end{equation}
A more general reflected BSDE of mean-field type with jumps has been considered in \cite{DDZ}. It should be pointed out that in \cite{DEH}, if the generator $f$ depends on $Z$ and the distribution of $Y$, it needs the monotonicity assumption with respect to the measure component (see Remark 5.2 (b) in \cite{DEH}). Although in \cite{DDZ}, the authors get rid of the monotonicity assumption, the generator $f$ can not depend on the distribution of $Z$. Moreover, in both papers, the barrier $l$ needs to satisfy the following condition: 

\begin{itemize}
    \item There exist two positive constants $\gamma_1,\gamma_2$, such that for any $y_1,y_2\in\mathbb{R}$, $\mu_1,\mu_2\in\mathcal{P}_2(\mathbb{R})$,
\begin{align*}
    |l(t,y_1,\mu_1)-l(t,y_2,\mu_2)|\leq \gamma_1|y_1-y_2|+\gamma_2W_2(\mu_1,\mu_2).
\end{align*}
\end{itemize}

The existence and uniqueness of the square-integrable solution to mean-field reflected BSDE \eqref{nonlinearyz'} require that the Lipschitz constant $\gamma_1$ and $\gamma_2$ are sufficiently small. More precisely, in \cite{DEH}, $\gamma_1$ and $\gamma_2$ should satisfy
\begin{align*}
    (\gamma_1+\gamma_2)^{\frac{1}{2}}(4\gamma_1+\gamma_2)^{\frac{1}{2}}<1,
\end{align*}
while in \cite{DDZ}, $\gamma_1$ and $\gamma_2$ should satisfy
\begin{align*}
    2(\gamma_1^2+\gamma_2^2)<1.
\end{align*}
\end{remark}


 \section{The case of nonlinear reflection}

The objective of the paper is to approximate the solution of mean-field mean reflected BSDE \eqref{nonlinearyz} by an interacting particle system. For this purpose, given a positive integer $N$, let $\{\xi^i\}_{1\leq i\leq N}$, $\{f^i\}_{1\leq i\leq N}$, $\{B^i\}_{1\leq i\leq N}$ be independent copies of $\xi,f$ and $B$, respectively. The augmented filtrations generated by $B^i$ and the family of  $\{B^i\}_{1\leq i\leq N}$ are denoted by $\mathbb{F}^i=\{\mathcal{F}^i_t\}_{t\in[0,T]}$ and $\mathbb{F}^{(N)}=\{\mathcal{F}^{(N)}_t\}_{t\in[0,T]}$, respectively. For any $0\leq s\leq t \leq T$, $\mathcal{T}^{N}_{s,t}$ is the collection of all $\mathbb{F}^{(N)}$-stopping times taking values in $[s,t]$. 

\begin{remark}\label{remarkxi}
 Suppose that the terminal value $\xi$ and the generator $f$ of mean-field mean reflected BSDE \eqref{nonlinearyz} take the following form:
\begin{align*}
\xi=G(\{B_t\}_{t\in[0,T]}), \ f(t,y,y',z,z')=F(t,\{B_{s\wedge t}\}_{s\in[0,T]},y,y',z,z'\}),
\end{align*}
where $G,F$ are measurable functions ensuring that (H1) and (H2) hold. Then, for any $1\leq i\leq N$, we may take
\begin{align*}
\xi^i=G(\{B^i_t\}_{t\in[0,T]}), \ f(t,y,z)=F(t,\{B^i_{s\wedge t}\}_{s\in[0,T]},y,y',z,z'\}).
\end{align*}
\end{remark}

In the following of this section, we suppose that $\xi,f,h$ satisfy conditions (H1)-(H5), where
\begin{itemize}
    \item [(H5)] $f$ is independent of $z,z'$ and $h$ is independent of $\omega$. 
\end{itemize}

\subsection{Well-posedness of the particle system}

Before introducing the particle system, for any positive integer $N$ and any $t\in[0,T]$, we first consider an operator $L^{(N)}_t:L^2(\mathcal{F}^{(N)}_t;\mathbb{R}^N)\rightarrow L^2(\mathcal{F}^{(N)}_t;\mathbb{R}^N)$ defined as follows: 
\begin{align*}
        L^{(N)}_t(X):=\inf\left\{x\geq 0:\frac{1}{N}\sum_{i=1}^N h(t,X^i+x)\geq 0\right\}, \ X=(X^1,\cdots,X^N).
    \end{align*}
This operator is of vital importance in establishing the existence of the solution to the particle system and the rate of convergence of the particle system towards the solution to mean-field mean reflected BSDE. We first introduce some properties of the operator $L^{(N)}_t(\cdot)$. 

\begin{proposition}\label{property of LN}
    (i) For any $X,Y\in L^2(\mathcal{F}^{(N)}_t;\mathbb{R}^N)$, we have 
    \begin{align}\label{LtXLtY}
        |L_t^{(N)}(X)-L_t^{(N)}(Y)|\le \frac{M}{m}\frac{1}{N}\sum_{j=1}^N|X^j-Y^j|.
    \end{align}
   
  \noindent  (ii) For any $S\in\mathcal{S}^2(\mathbb{F}^{(N)};\mathbb{R}^N)$, we have $\{L^{(N)}_t(S_t)\}_{t\in[0,T]}\in \mathcal{S}^2(\mathbb{F}^{(N)})$.
\end{proposition}

\begin{proof}
   By the proof of Theorem 3.1 in \cite{BH} (see Eq. (4)), we obtain \eqref{LtXLtY} directly. 
Let us set $x_t=\inf\{x\geq 0:h(t,x)\geq 0\}$. By Eq. \eqref{LtXLtY}, we have 
    \begin{align}\label{LtSt}
        |L^{(N)}_t(S_t)|\leq |L^{(N)}_t(S_t)-L^{(N)}_t(0)|+|L^{(N)}_t(0)|\leq x_t+\frac{M}{m}\frac{1}{N}\sum_{j=1}^N|S^j_t|.
    \end{align}
    Noting that $\{x_t\}_{t\in[0,T]}$ is a continuous function, it follows that 
    \begin{align*}
        \E\left[\sup_{t\in[0,T]}|L^{(N)}_t(S_t)|^2\right]<\infty.
    \end{align*}
   It is clear that $\{L^{(N)}_t(S_t)\}_{t\in[0,T]}$ is $\mathbb{F}^{(N)}$-adapted. It remains to prove that it is continuous. To this end, given $y=(y^1,\cdots,y^N)\in\mathbb{R}^N$, we define 
  \begin{align*}
      h^{(N)}_y(t,x)=\frac{1}{N}\sum_{i=1}^N h(t,y^i+x).
  \end{align*}
  It is easy to check that, for any $x\in\mathbb{R}$, $h^{(N)}_y(\cdot,x)$ is continuous and for any $t\in[0,T]$, $h^{(N)}_y(t,\cdot)$ is strictly increasing and bi-Lipschitz, i.e., for any $x,x'\in\mathbb{R}$,
  \begin{align*}
      m|x-x'|\leq |h^{(N)}_y(t,x)-h^{(N)}_y(t,x')|\leq M |x-x'|.
  \end{align*}
  Therefore, for any $X\in L^2(\mathcal{F}^{(N)}_t;\mathbb{R}^N)$, the following equation admits a unique solution $\tilde{L}^{(N)}_t(X)$
  \begin{align}\label{tildeLN}
        \frac{1}{N}\sum_{i=1}^N h(t,X^i+\tilde{L}^{(N)}_t(X))= 0.
    \end{align}
  Moreover, we have $L^{(N)}_t(X)=(\tilde{L}^{(N)}_t(X))^+$. For any $0\leq s\leq t\leq T$, simple calculation yields that 
    \begin{equation}\label{LtXLsX}\begin{split}
        |L^{(N)}_t(S_s)-L^{(N)}_s(S_s)|\leq &|\tilde{L}^{(N)}_t(S_s)-\tilde{L}^{(N)}_s(S_s)|\\
        \leq& \frac{1}{m}|h^{(N)}_{S_s}(t,\tilde{L}^{(N)}_t(S_s))-h^{(N)}_{S_s}(t,\tilde{L}^{(N)}_s(S_s))|\\
        =& \frac{1}{m}|h^{(N)}_{S_s}(s,\tilde{L}^{(N)}_s(S_s))-h^{(N)}_{S_s}(t,\tilde{L}^{(N)}_s(S_s))|.
    \end{split}\end{equation}
Combining Eqs. \eqref{LtXLtY} and \eqref{LtXLsX}, the process $\{L^{(N)}_t(S_t)\}_{t\in[0,T]}$ is continuous. Finally, we obtain the desired result.
\end{proof}

In this subsection, we consider the interacting particle system taking the following form:
\begin{equation}\label{eq7}
\begin{cases}
Y^i_t=\theta^i+\int_t^T f^i\left(s,Y^i_s,\frac{1}{N}\sum_{j=1}^N Y^j_s\right) ds-\int_t^T\sum_{j=1}^N Z^{i,j}_sdB^j_s+K^{(N)}_T-K^{(N)}_t, \ \forall 1\leq i\leq N,\\
 \frac{1}{N}\sum_{i=1}^N h(t,Y^i_t)\geq 0, \ t\in[0,T], \textrm{ and }
\int_0^T\frac{1}{N}\sum_{i=1}^N h(t,Y^i_t)dK^{(N)}_t=0,
\end{cases}
\end{equation}
where 
\begin{align*}
    \theta^i=\xi^i+L^{(N)}_T(\xi^{(N)}), \ \xi^{(N)}=(\xi^1,\cdots,\xi^N).
\end{align*}
Actually, this particle system can be viewed as a multi-dimensional reflected BSDE. As claimed in \cite{BH}, the terminal value of \eqref{eq7} cannot be taken as $(\xi^1,\cdots,\xi^N)$, since under the assumption that $\E[h(T,\xi)]\geq 0$, we do not have in general
\begin{align*}
    \frac{1}{N}\sum_{i=1}^N h(T,\xi^i)\geq 0.
\end{align*}
However, by the definition of $L^{(N)}_T$, we have 
\begin{align*}
    \frac{1}{N}\sum_{i=1}^N h(T,\theta^i)\geq 0.
\end{align*}
Now, we state the main result in this subsection, which generalized Proposition 4.1 in \cite{BH} to the mean-field setting. 


\begin{theorem}\label{thm4.1}
The particle system \eqref{eq7} has a unique solution $(\{Y^i,Z^i\}_{1\leq i\leq N},K^{(N)})$ with $K^{(N)}\in \mathcal{A}^2(\mathbb{F}^{(N)})$, $Y^i\in \mathcal{S}^2(\mathbb{F}^{(N)})$ and $Z^i\in \mathcal{H}^2(\mathbb{F}^{(N)};\mathbb{R}^N)$ for $i=1,\cdots,N$. 
\end{theorem}

\begin{proof}
First, let us briefly introduce how to construct the solution to multi-dimensional reflected BSDE \eqref{eq7} when $f^i$ is independent of $y,y'$ as shown in the proof of Theorem 3.1 in \cite{BH}. 
For each $i=1,\cdots,N$, we define
\begin{align*}
    U^i_t=\E\left[\xi^i+\int_t^T f^i(s) ds\Big|\mathcal{F}^{(N)}_t\right], \ t\in[0,T]. 
\end{align*}
By Proposition \ref{property of LN}, we have $\{L^{(N)}_t(U_t)\}_{t\in[0,T]}\in\mathcal{S}^2(\mathbb{F}^{(N)})$.  Let $S^{(N)}$ be the Snell envelope of the process $\{L^{(N)}_t(U_t)\}_{t\in[0,T]}$, i.e., 
\begin{align*}
    S^{(N)}_t=\esssup_{\tau\in \mathcal{T}^N_{t,T}}\E\left[L^{(N)}_\tau(U_\tau)\big|\mathcal{F}^{(N)}_t\right].
\end{align*}
Since $S$ is an $\mathbb{F}^{(N)}$-supermartingale with $S\in \mathcal{S}^2(\mathbb{F}^{(N)})$, it has the following decomposition
\begin{align*}
    S^{(N)}_t=M^{(N)}_t-K^{(N)}_t,
\end{align*}
where $K^{(N)}\in \mathcal{A}^2(\mathbb{F}^{(N)})$. Now, set 
\begin{align*}
    Y^i_t=U^i_t+S_t
\end{align*}
and let $Z^i\in \mathcal{H}^2(\mathbb{F}^{(N)};\mathbb{R}^N)$ be obtained by the following martingale representation 
\begin{align*}
    \E\left[\theta^i+\int_0^T f^i(s) ds+K^{(N)}_T\Big|\mathcal{F}^{(N)}_t\right]=\E\left[\theta^i+\int_0^T f^i(s) ds+K^{(N)}_T\right]+\int_0^t\sum_{j=1}^N Z^{i,j}_sdB^j_s.
\end{align*}
Then, $(\{Y^i,Z^i\}_{1\leq i\leq N},K^{(N)})$ is the solution to \eqref{eq7}.

Now, we are in a position to prove the main result by a fixed point argument when the time horizon is sufficiently small. Given $V\in \mathcal{S}^2(\mathbb{F}^{(N)};\mathbb{R}^N)$, let $(\{Y^i,Z^i\}_{1\leq i\leq N},K^{(N)})$ stand for the unique solution to the following equation
\begin{equation*}
\begin{cases}
Y^i_t=\theta^i+\int_t^T f^i\left(s,V^i_s,\frac{1}{N}\sum_{j=1}^N V^j_s\right) ds-\int_t^T\sum_{j=1}^N Z^{i,j}_sdB^j_s+K^{(N)}_T-K^{(N)}_t, \ \forall 1\leq i\leq N,\\
 \frac{1}{N}\sum_{i=1}^N h(Y^i_t)\geq 0, \ t\in[0,T], \textrm{ and }
\int_0^T\frac{1}{N}\sum_{i=1}^N h(Y^i_t)dK^{(N)}_t=0.
\end{cases}
\end{equation*}
We define the map $\Gamma:\mathcal{S}^2(\mathbb{F}^{(N)};\mathbb{R}^N)\rightarrow \mathcal{S}^2(\mathbb{F}^{(N)};\mathbb{R}^N)$ as follows
\begin{align*}
    \Gamma(V):=Y.
\end{align*}
Given another $\tilde{V}\in \mathcal{S}^2(\mathbb{F}^{(N)};\mathbb{R}^N)$, let  $\tilde{Y}=\Gamma (\tilde{V})$. We define
\begin{align*}
&U^i_t:=\E\left[\xi^i+\int_t^T f^i\left(s,V^i_s,\frac{1}{N}\sum_{j=1}^N V^j_s\right)ds\Bigg|\mathcal{F}^{(N)}_t\right], \\ &\tilde{U}^i_t:=\E\left[\xi^i+\int_t^T f^i\left(s,\tilde{V}^i_s,\frac{1}{N}\sum_{j=1}^N \tilde{V}^j_s\right)ds\Bigg|\mathcal{F}^{(N)}_t\right].
\end{align*}
Set $\Delta P_t=P_t-\tilde{P}_t$ for $P=Y,U,V$ and $\Delta L^{(N)}_t=L^{(N)}_t(U_t)-L^{(N)}_t(\tilde{U}_t)$. Then, it is easy to check that 
\begin{align*}
    &|\Delta Y^i_t|\leq |\Delta U^i_t|+\E\left[\sup_{s\in[0,T]}|\Delta L^{(N)}_s|\Big|\mathcal{F}^{(N)}_t\right] \\
    \leq& LT\E\left[\sup_{s\in[0,T]}|\Delta V^i_s|\Big|\mathcal{F}^{(N)}_t\right]+LT\E\left[\frac{1}{N}\sup_{s\in[0,T]}\sum_{j=1}^N|\Delta V^j_s|\Big|\mathcal{F}^{(N)}_t\right]+\E\left[\sup_{s\in[0,T]}|\Delta L^{(N)}_s|\Big|\mathcal{F}^{(N)}_t\right].
\end{align*}
Applying Doob's inequality and using the fact that (see Eq. \eqref{LtXLtY})
\begin{align*}
|\Delta L^{(N)}_t|\leq \frac{M}{mN}\sum_{j=1}^N|\Delta U^j_t|,
\end{align*}
we obtain that 
\begin{equation}\begin{split}\label{Delta Yit}
    \E\left[\sup_{t\in[0,T]}|\Delta Y^i_t|^2\right]\leq &12L^2T^2\E\left[\sup_{t\in[0,T]}|\Delta V^i_t|^2\right]+\frac{12M^2}{m^2N^2}\E\left[\left(\sup_{t\in[0,T]}\sum_{j=1}^N|\Delta U^j_t|\right)^2\right]\\
    &+\frac{12 L^2T^2}{N^2}\E\left[\left(\sup_{s\in[0,T]}\sum_{j=1}^N|\Delta V^j_s|\right)^2\right].
\end{split}\end{equation}
It follows from the Lipschitz assumption for $f$ that 
\begin{align*}
    |\Delta U^i_t|\leq LT\E\left[\sup_{s\in[0,T]}|\Delta V^i_s|\Big|\mathcal{F}^{(N)}_t\right]+LT\E\left[\frac{1}{N}\sup_{s\in[0,T]}\sum_{j=1}^N|\Delta V^j_s|\Big|\mathcal{F}^{(N)}_t\right].
\end{align*}
Summing over $i$ yields that 
\begin{align*}
    \sum_{i=1}^N |\Delta U^i_t|\leq 2LT\E\left[\sup_{s\in[0,T]}\sum_{j=1}^N|\Delta V^j_s|\Big|\mathcal{F}^{(N)}_t\right].
\end{align*}
Therefore, we have
\begin{align*}
    \E\left[\left(\sup_{t\in[0,T]}\sum_{j=1}^N|\Delta U^j_t|\right)^2\right]\leq 16L^2T^2 \E\left[\left(\sup_{s\in[0,T]}\sum_{j=1}^N|\Delta V^j_s|\right)^2\right]\leq 16L^2T^2N \E\left[\sup_{s\in[0,T]}\sum_{j=1}^N|\Delta V^j_s|^2\right]. 
\end{align*}
Plugging the above inequality into \eqref{Delta Yit} implies that 
\begin{equation*}\begin{split}
    \E\left[\sup_{t\in[0,T]}|\Delta Y^i_t|^2\right]\leq 12L^2T^2\E\left[\sup_{t\in[0,T]}|\Delta V^i_t|^2\right]+12L^2T^2\left(1+16\frac{M^2}{m^2}\right)\E\left[\frac{1}{N}\sup_{t\in[0,T]}\sum_{j=1}^N|\Delta V^j_t|^2\right].
\end{split}\end{equation*}
Summing over $i$, we finally obtain that 
\begin{equation*}\begin{split}
    \E\left[\frac{1}{N}\sum_{i=1}^N\sup_{t\in[0,T]}|\Delta Y^i_t|^2\right]\leq 24L^2T^2\left(1+8\frac{M^2}{m^2}\right)\E\left[\frac{1}{N}\sup_{t\in[0,T]}\sum_{j=1}^N|\Delta V^j_t|^2\right].
\end{split}\end{equation*}
Choosing $T\leq \varepsilon$ with $\varepsilon>0$ be such that 
\begin{align*}
24L^2 \varepsilon^2\left(1+8\frac{M^2}{m^2}\right)\leq \frac{1}{2},
\end{align*}
then $\Gamma$ is a contraction mapping. Hence, $\Gamma$ has a unique fixed point in $\mathcal{S}^2(\mathbb{F}^{(N)};\mathbb{R}^N)$ when $T\leq \varepsilon$, i.e., there exists a unique $\{Y^i\}_{1\leq i\leq N}$ solving \eqref{eq7} for some $(\{Z^i\}_{1\leq i\leq N},K^{(N)})$ with $K^{(N)}\in \mathcal{A}^2(\mathbb{F}^{(N)})$ and $Z^i\in \mathcal{H}^2(\mathbb{F}^{(N)};\mathbb{R}^N)$ for $i=1,\cdots,N$,  on small time interval $[0,T]$. Moreover, since $\{Y^i\}_{1\leq i\leq N}$ is unique, $\{Z^i\}_{1\leq i\leq N}$ is unique by applying It\^{o}'s formula and finally $K^{(N)}$ is also unique.

For the general case, let $n$ be a positive integer such that $T/n<\varepsilon$. For any $k=0,1,\cdots,n$, set $t_k=\frac{kT}{n}$. For $k=n,n-1,\cdots,1$, let  $(\{Y^{i,k},Z^{i,k}\}_{1\leq i\leq N},, K^{(N),k})$ be the unique solution to \eqref{eq7} on time interval $[t_{k-1},t_k]$ with $K^{(N),k}_{t_{k-1}}=0$ and terminal value 
\begin{align*}
    \xi^{i,k}=\begin{cases} Y^{i,k+1}_{t_k}, &k=n-1,\cdots,1,\\ 
    \theta^i, &k=n.
    \end{cases}
\end{align*}
Set
\begin{align*}%
&Y^i_t=Y^{i,k}_t, \ Z^i_t=Z^{i,k}_t,\ K^{(N)}_t=K^{(N),k}_t+\sum_{l<k}K^{(N),l}_{t_l}, \ t\in[t_{k-1},t_k], \ k=1,2,\cdots,n.
\end{align*}
Then, $(\{Y^i,Z^i\}_{1\leq i\leq N}, K^{(N)})$ is the unique solution to \eqref{eq7}. The proof is complete. 
\end{proof}

The following proposition provides a priori estimates for the solution to \eqref{eq7}.
\begin{proposition}\label{prop4.2}
There exists a constant $C$ independent of $N$, such that, for all $1\leq i\leq N$,
\begin{align}\label{estimateYiZiKN}
\E\left[\sup_{t\in[0,T]}|Y^i_t|^2+\int_0^T|Z^i_s|^2ds+|K^{(N)}_T|^2\right]
\leq& C\left(1+\E\left[|\xi|^2\right]+\E\left[\int_0^T|f(t,0,0)|^2dt\right]\right).
\end{align}
\end{proposition}

\begin{proof}
Set 
\begin{align*}
\bar{U}^i_t=\E\left[\xi^i+\int_t^T f^i\left(s,Y^{i}_s,\frac{1}{N}\sum_{j=1}^N Y^j_s\right)ds\Bigg|\mathcal{F}^{(N)}_t\right].
\end{align*}
Let $\bar{S}$ be the Snell envelope for $\{L^{(N)}_t(\bar{U}_t)\}_{t\in[0,T]}$. By the proof of Theorem \ref{thm4.1}, we have
\begin{align}\label{rep Yi}
Y^i_t=\bar{U}^i_t+\bar{S}_t=\bar{U}^i_t+\esssup_{\tau\in\mathcal{T}^N_{t,T}}\E\left[L^{(N)}_\tau(\bar{U}_\tau)\Big|\mathcal{F}^{(N)}_t\right].
\end{align}
It follows that, for any $t\in[0,T]$, 
\begin{equation}\begin{split}\label{yit}
|Y^i_t|\leq &\E\left[|\xi^i|+\int_t^T |f^i(s,0,0)|ds+L\int_t^T|Y^i_s|ds\Big|\mathcal{F}^{(N)}_t \right]\\
&+\E\left[\frac{L}{N}\int_t^T\sum_{j=1}^N |Y^j_s|ds\Big|\mathcal{F}^{(N)}_t\right]+\E\left[\sup_{s\in[t,T]}|L^{(N)}_s(\bar{U}_s)|\Big|\mathcal{F}^{(N)}_t\right].
\end{split}\end{equation}
Recalling Eq. \eqref{LtSt}, we have
\begin{align*}
|L^{(N)}_t(\bar{U}_t)|\leq |x_t|+\frac{M}{mN}\sum_{j=1}^N|\bar{U}^j_t|.
\end{align*}
where $x_t=\inf\{x\geq 0:h(t,x)\geq 0\}$. Applying Doob's inequality and H\"{o}lder's inequality, there exists a constant $C$ independent of $N$, such that 
\begin{equation}\label{barpsin}\begin{split}
&\E\left[\sup_{s\in[t,T]}|L^{(N)}_s(\bar{U}_s)|^2\right]\leq C\left(1+\E\left[|\xi|^2+\int_0^T |f(s,0,0)|^2ds\right]+\E\left[\int_t^T\frac{1}{N}\sum_{j=1}^N|Y^j_s|^2ds\right]\right).
\end{split}\end{equation}
By Eqs. \eqref{yit} and \eqref{barpsin}, we have
\begin{align}\label{yit'}
\E[|Y^i_t|^2]\leq C\left(1+\E\left[|\xi|^2+\int_0^T |f(s,0,0)|^2ds\right]+\E\left[\int_t^T\frac{1}{N}\sum_{j=1}^N|Y^j_s|^2ds\right]+\E\left[\int_t^T |Y^i_s|^2ds\right]\right).
\end{align}
Summing over $i$ implies that
\begin{align*}
\E\left[\frac{1}{N}\sum_{i=1}^N|Y^i_t|^2\right]\leq C\left(1+\E\left[|\xi|^2+\int_0^T |f(s,0,0)|^2ds\right]+\int_t^T\E\left[\frac{1}{N}\sum_{j=1}^N|Y^j_s|^2\right]ds\right).
\end{align*}
It follows from the Gronwall inequality that 
\begin{align}\label{mean of Yi}
\E\left[\frac{1}{N}\sum_{i=1}^N|Y^i_t|^2\right]\leq C\left(1+\E\left[|\xi|^2+\int_0^T |f(s,0,0)|^2ds\right]\right).
\end{align}
Plugging the above inequality into \eqref{barpsin}, \eqref{yit'} and using the Gronwall inequality again, we obtain that 
\begin{equation}\label{barpsinyit}\begin{split}
\E\left[\sup_{s\in[0,T]}|L^{(N)}_s(\bar{U}_s)|^2\right]&\leq C\left(1+\E\left[|\xi|^2+\int_0^T |f(s,0,0)|^2ds\right]\right),\\
\E\left[|Y^i_t|^2\right]&\leq C\left(1+\E\left[|\xi|^2+\int_0^T |f(s,0,0)|^2ds\right]\right).
\end{split}\end{equation}
Finally, combining \eqref{yit}, \eqref{mean of Yi} and \eqref{barpsinyit}, applying Doob's inequality and H\"{o}lder's inequality, we obtain 
\begin{align}\label{estimateYi}
\E\left[\sup_{t\in[0,T]}|Y^i_t|^2\right]
\leq& C\left(1+\E\left[|\xi|^2\right]+\E\left[\int_0^T|f(t,0,0)|^2dt\right]\right).
\end{align}

Note that for each $1\leq i\leq N$, we have
   \begin{align*}
       K^{(N)}_t=Y^i_0-Y^i_T-\int_0^T f^i\left(s,Y^i_s,\frac{1}{N}\sum_{i=1}^N Y^i_s\right) ds+\int_0^T\sum_{j=1}^N Z^{i,j}_sd B^j_s.
   \end{align*}
   Then, there exists a constant $c_1$ independent of $N$, such that  
   \begin{align}\label{estiKN}
       \E\left[|K^{(N)}_T|^2\right]\leq c_1\E\left[\sup_{t\in[0,T]}|Y^i_t|^2+\int_0^T|f^i(s,0,0)|^2ds+\int_0^T|Z^i_s|^2ds+\frac{1}{N}\int_0^T\sum_{i=1}^N|Y^i_s|^2ds\right].
   \end{align}
   Applying It\^{o}'s formula to $|Y^i_t|^2$ and recalling that $\theta^i=\xi^i+L^{(N)}_T(\xi)=\xi^i+L^{(N)}_T(\bar{U}_T)$, we obtain 
   \begin{align*}
       \E\left[\int_0^T |Z^i_t|^2dt\right]\leq &\E[|\theta^i|^2]+2\E\left[\int_0^T Y^i_s f^i\left(s,Y^i_s,\frac{1}{N}\sum_{i=1}^N Y^i_s\right) ds\right]+2\E\left[\int_0^T Y^i_s dK^{(N)}_s\right]\\
       \leq & 2\E[|\xi^i|^2]+2\E\left[|L^{(N)}_T(\bar{U}_T)|^2\right]+\E\left[\int_0^T|f^i(s,0,0)|^2ds\right]+\E\left[\int_t^T\frac{1}{N}\sum_{i=1}^N|Y^i_s|^2ds\right]\\
       &+\left(2c_1+(1+2L+L^2)T\right)\E\left[\sup_{t\in[0,T]}|Y^i_t|^2\right]+\frac{1}{2c_1}\E\left[|K^{(N)}_T|^2\right].
   \end{align*}
   Plugging Eq. \eqref{estiKN} into the above inequality and using Eqs. \eqref{mean of Yi}-\eqref{estimateYi},  we obtain that 
   \begin{align}\label{estiZi}
       \E\left[\int_0^T |Z^i_t|^2dt\right]
\leq C\left(1+\E\left[|\xi|^2\right]+\E\left[\int_0^T|f(t,0,0)|^2dt\right]\right).
   \end{align}
   Finally, plugging Eqs. \eqref{mean of Yi}, \eqref{estimateYi} and \eqref{estiZi} into Eq. \eqref{estiKN}, we obtain the desired result.
\end{proof}

\begin{remark}\label{rem distribution}
Suppose that $f$ satisfies (H2') (see Remark \ref{rem MFMRBSDE}) and is independent of $z$ and $\nu$. Consider the following particle system
\begin{equation*}
\begin{cases}
Y^i_t=\theta^i+\int_t^T f^i\left(s,Y^i_s,\frac{1}{N}\sum_{j=1}^N \delta_{Y^j_s}\right) ds-\int_t^T\sum_{j=1}^N Z^{i,j}_sdB^j_s+K^{(N)}_T-K^{(N)}_t, \ \forall 1\leq i\leq N,\\
 \frac{1}{N}\sum_{i=1}^N h(t,Y^i_t)\geq 0, \ t\in[0,T], \textrm{ and }
\int_0^T\frac{1}{N}\sum_{i=1}^N h(t,Y^i_t)dK^{(N)}_t=0,
\end{cases}
\end{equation*}
By a similar analysis as the proof of Theorem \ref{thm4.1}, it has a unique solution $(\{Y^i,Z^i\}_{1\leq i\leq N},K^{(N)})$ with $K^{(N)}\in \mathcal{A}^2(\mathbb{F}^{(N)})$, $Y^i\in \mathcal{S}^2(\mathbb{F}^{(N)})$ and $Z^i\in \mathcal{H}^2(\mathbb{F}^{(N)};\mathbb{R}^N)$ for $i=1,\cdots,N$. Moreover, the same result as in Proposition \ref{prop4.2} still holds.
\end{remark}

\subsection{Propagation of chaos}

In this subsection, we use the solution of the interacting particle system \eqref{eq7} to approximate the solution of mean-field BSDE with mean reflection \eqref{nonlinearyz} and establish the rate of convergence when the generator $f$ does not depends on $z,z'$. For this purpose, let $\xi$, $f$ (independent of $z$), $\{\xi^i\}_{1\leq i\leq N}$, $\{f^i\}_{1\leq i\leq N}$ be given as in Remark \ref{remarkxi}. Let $(\bar{Y}^i,\bar{Z}^i,K)$ be the solution to the following mean-field mean reflected BSDE:
\begin{displaymath}
\begin{cases}
\bar{Y}^i_t=\xi^i+\int_t^T f^i(s,\bar{Y}^i_s,\E[\bar{Y}^i_s])ds-\int_t^T \bar{Z}^i_s dB^i_s+K_T-K_t, \\
\E[h(t,\bar{Y}^i_t)]\geq 0, t\in[0,T] \textrm{ and} \int_0^T \E[h(t,\bar{Y}^i_t)]dK_t=0.
\end{cases}
\end{displaymath}
Clearly, $(\bar{Y}^i,\bar{Z}^i,K)$, $1\leq i\leq N$ are independent copies of $(Y,Z,K)$, the solution to \eqref{nonlinearyz}.

We define
\begin{align*}
\widetilde{U}^i_t=\E\left[\xi^i+\int_t^T\bar{f}^i_sds \bigg|\mathcal{F}^i_t\right],
\end{align*}
where 
\begin{align*}
  \bar{f}^i_s=f^i\left(s,\bar{Y}^i_s,\E[\bar{Y}^i_s]\right).  
\end{align*}
 Since the Brownian motions $\{B^i\}_{1\leq i\leq N}$ are independent, we have  
\begin{align*}
\widetilde{U}^i_t=\E\left[\xi^i+\int_t^T\bar{f}^i_sds \bigg|\mathcal{F}^{(N)}_t\right].
\end{align*}
For each $t\in[0,T]$, let $L_t(\widetilde{U}^i_t)\in [0,\infty)$ be defined as follows
\begin{align*}
    L_t(\widetilde{U}^i_t):=\inf\{x\geq 0: \E[h(t,x+\widetilde{U}^i_t)]\geq0\}.
\end{align*}
For any $X\in L^2(\mathcal{F}^{(N)}_T)$, it is easy to check that the mapping $x\mapsto\E[h(t,x+X)]$ is a one-to-one correspondence. Therefore, the following equation
\begin{align}\label{tildeL}
    \E[h(t,x+\widetilde{U}^i_t)]=0
\end{align}
admits a unique solution, which is denoted by $\tilde{L}_t(\widetilde{U}^i_t)$.
By the proof of Proposition 7 in \cite{BEH}, we have 
\begin{align*}
    K_T-K_t=\sup_{s\in[t,T]}L_s(\widetilde{U}^i_s)=\sup_{s\in[t,T]}(\tilde{L}_s(\widetilde{U}^i_s))^+.
\end{align*}
It follows that 
\begin{align}\label{baryi}
\bar{Y}^i_t=\E\left[\xi^i+\int_t^T\bar{f}^i_sds \bigg|\mathcal{F}^{(N)}_t\right]+K_T-K_t=\widetilde{U}^i_t+\esssup_{\tau\in \mathcal{T}^N_{t,T}}\E\left[L_\tau(\widetilde{U}^i_\tau)|\mathcal{F}^{(N)}_t\right].
\end{align}
Before introducing the main result in this subsection, we first provide the following technical lemma, whose proof may be referred to the one of Lemma 5.2 in \cite{LS'}.

\begin{lemma}\label{Lip}
    Suppose that $h\in C^{1,2}_b([0,T]\times\mathbb{R})$ and $\sup_{t\in[0,T]}\E[|\bar{Z}^1_t|^2]<\infty$. Then, the map $t\mapsto \tilde{L}_t(\widetilde{U}^i_t)$ is Lipschitz continuous. 
\end{lemma}

\begin{remark}
(i) Fix $p\geq 2$. Given $\xi\in L^p(\mathcal{F}_T)$ (i.e., $\E[|\xi|^p]<\infty$) and $f$ satisfying (H2) and  
\begin{align*}
    \E\left[\int_0^T|f(t,0,0,0,0)|^pdt\right]<\infty,
\end{align*}
consider the following mean-field BSDE
\begin{align}\label{MFBSDE}
    Y_t=\xi+\int_t^T f(s,Y_s,\mathbf{E}[Y_s],Z_s,\mathbf{E}[Z_s])ds-\int_t^T Z_s dB_s. 
\end{align}
Moreover, suppose that 
\begin{itemize}
    \item[(1)] $f$ is continuously differentiable in $y$ and $z$ with uniformly bounded derivative. 
    \item[(2)] $\xi$ and $f(\cdot, y,y',z,z')$ are Malliavin differentiable for each $y,y',z,z'$ with 
    \begin{itemize}
        \item [(i)] $\sup_{\theta \in[0,T]}\mathbf{E}\left[|D_{\theta}\xi|^p\right]<\infty$;
        \item [(ii)] $\sup_{\theta \in[0,T]}\mathbf{E}\left[\left(\int_{0}^{T}|D_{\theta}f(t,Y_t,\mathbf{E}[Y_t],Z_t,\mathbf{E}[Z_t])|dt\right)^{p}\right]<\infty$.
    \end{itemize}   
\end{itemize}
Then, we have $\mathbf{E}\left[\sup_{t\in[0,T]}\left|Z_t\right|^p\right]<\infty$ (see Lemma 5.3 in \cite{LS'}).  

\noindent(ii) Given a deterministic function $A$, suppose that $(Y,Z)$ satisfies the following equation \begin{align*}
    Y_t=\xi+\int_t^T f(s,Y_s,\mathbf{E}[Y_s],Z_s,\mathbf{E}[Z_s])ds-\int_t^T Z_s dB_s+A_T-A_t.
\end{align*}
Under the same assumptions as in (i), we have $\mathbf{E}\left[\sup_{t\in[0,T]}\left|Z_t\right|^p\right]<\infty.$

\noindent (iii) Suppose that the parameters for the mean-field BSDE \eqref{MFBSDE} are given by  $\xi=g(X_T)$ and  $f(t,y,y',z,z')=F(X_t,y,y',z,z')$, where $X_t$ satisfies the following SDE:
    \begin{equation*}
        X_t=x_0+\int_{0}^{t}b(X_s)ds+\int_{0}^{t}\sigma(X_s)dB_s,
    \end{equation*}
 and the functions $b,\sigma,F,g$ are deterministic and continuously differentiable such that $\partial_xb$, $\partial_{x}\sigma$,  $\partial_{x}g$, $\partial_{x}F$,  $\partial_{y}F$, $\partial_{z}F$ are bounded. Then, the assumptions in (i) are satisfied. Therefore, we can derive 
  $\mathbf{E}\left[\sup_{t\in[0,T]}\left|Z_t\right|^p\right]<\infty.$ 
 \end{remark}

Now, we are ready to present the main result in this subsection.
\begin{theorem}\label{thm4.3}
For any $1\leq i\leq N$, set 
\begin{align*}
\hat{Y}^i:=Y^i-\bar{Y}^i, \ \hat{K}:=K^{(N)}-K,\ \hat{Z}^i:=Z^i-\bar{Z}^i e_i,
\end{align*}
where $(e_1,\cdots,e_N)$ is the canonical basis in $\mathbb{R}^N$.

\begin{itemize}
\item[1.] If $h\in C^{1,2}_b([0,T]\times\mathbb{R})$ and $\sup_{t\in[0,T]}\E[|\bar{Z}^1_t|^4]<\infty$, then there exists a constant $C$ independent of $N$, such that
\begin{align*}
\E\left[\sup_{t\in[0,T]}|\hat{Y}^i_t|^2\right]\leq CN^{-1}, \ \E\left[\sup_{t\in[0,T]}|\hat{K}_t|^2\right]\leq CN^{-1/2}, \ \E\left[\int_0^T|\hat{Z}^i_t|^2dt\right]\leq CN^{-1/2}.
\end{align*}
\item[2.] If $\xi$ is $p$-integrable, $\E[\int_0^T |f(t,0,0)|^pdt]<\infty$ and $\sup_{t\in[0,T]}\E[|\bar{Z}^1_t|^p]<\infty$ for some $p>4$, then there exists a constant $C$ independent of $N$, such that
\begin{align*}
\E\left[\sup_{t\in[0,T]}|\hat{Y}^i_t|^2\right]\leq CN^{-1/2}, \ \E\left[\sup_{t\in[r,T]}|\hat{K}_t|^2\right]\leq CN^{-1/4}, \ \E\left[\int_0^T|\hat{Z}^i_t|^2dt\right]\leq CN^{-1/4}.
\end{align*}
\end{itemize}
\end{theorem}

\begin{proof}
  \textbf{Step 1.} We first show that in both cases, there exists a constant $C$ independent of $N$, such that 
  \begin{align*}
      &\E\left[\int_0^T|\hat{Z}^i_t|^2dt\right]\leq C\left(\E\left[\sup_{t\in[0,T]}|\hat{Y}^i_t|^2\right]\right)^{\frac{1}{2}}, \\  
      &\E\left[\sup_{t\in[0,T]}|\hat{K}_t|^2\right]\leq C\left\{\left(\E\left[\sup_{t\in[0,T]}|\hat{Y}^i_t|^2\right]\right)^{\frac{1}{2}}+\frac{1}{N}\sum_{j=1}^N \E\left[\sup_{t\in[0,T]}|\hat{Y}^j_t|^2\right]+\frac{1}{N}\right\}.
  \end{align*}
    It is easy to check that 
    \begin{align*}
        \hat{Y}_t^i=&\hat{Y}^i_T+\int_t^T \hat{f}^i_sds
        -\int_t^T \sum_{j=1}^N \hat{Z}^{i,j}_s dB^j_s+\hat{K}_T-\hat{K}_t,
    \end{align*}
    where 
    \begin{align*}
        \hat{f}^i_s=f^i\left(s,Y^i_s,\frac{1}{N}\sum_{j=1}^N Y^j_s\right)-f^i(s,\bar{Y}^i_s,\E[\bar{Y}^i_s]).
    \end{align*}
    Applying It\^{o}'s formula to $|\hat{Y}^i_t|^2$ and then taking expectations, we obtain that 
    \begin{align*}
      \E\left[\int_0^T|\hat{Z}^i_t|^2dt\right]\leq&\E\left[|\hat{Y}^i_T|^2+2\int_0^T \hat{Y}^i_s \hat{f}^i_s ds+2\int_0^T\hat{Y}^i_s d\hat{K}_s\right]\\
      \leq & C\E\left[\sup_{t\in[0,T]}|\hat{Y}^i_s|\int_0^T \left(|f^i(s,0,0)|+\left|\frac{1}{N}\sum_{j=1}^N Y^j_s\right|+\E[|\bar{Y}^i_s|]\right)ds\right]\\
      &+C\E\left[|\hat{Y}^i_T|^2+\sup_{t\in[0,T]}|\hat{Y}^i_s|^2+\sup_{t\in[0,T]}|\hat{Y}^i_s|(|K^{(N)}_T|+|K_T|)\right].
    \end{align*}
    Noting that $\bar{Y}^i$ is an independent copy of $Y\in\mathcal{S}^2$ and recalling Eqs. \eqref{estimateYiZiKN} and \eqref{mean of Yi}, we obtain that 
    \begin{align*}
      \E\left[\int_0^T|\hat{Z}^i_t|^2dt\right]\leq C\left(\E\left[\sup_{t\in[0,T]}|\hat{Y}^i_t|^2\right]\right)^{\frac{1}{2}}.
  \end{align*}
  Due to the fact that 
  \begin{align*}
        \hat{K}_t=&\hat{Y}^i_0-\hat{Y}^i_t-\int_0^t \hat{f}^i_sds
        +\int_0^t \sum_{j=1}^N \hat{Z}^{i,j}_s dB^j_s,
    \end{align*}
    it follows from H\"{o}lder's inequality and Doob's inequality that 
    \begin{align*}
       \E\left[\sup_{t\in[0,T]}|\hat{K}_t|^2\right]\leq C\left\{\E\left[\sup_{t\in[0,T]}|\hat{Y}^i_t|^2\right]+\E\left[\int_0^T|\hat{Z}^i_t|^2dt\right]+\E\left[\int_0^T|\hat{f}^i_s|^2ds\right]\right\}. 
    \end{align*}
    Simple calculation yields that 
    \begin{equation}\begin{split}\label{hatfi}
        |\hat{f}^i_s|&=\left|f^i\left(s,Y^i_s,\frac{1}{N}\sum_{j=1}^N Y^j_s\right)-f^i\left(s,\bar{Y}^i_s,\frac{1}{N}\sum_{j=1}^N\E[\bar{Y}^j_s]\right)\right|\\
        &\leq L\left\{|\hat{Y}^i_s|+\frac{1}{N}\sum_{j=1}^N|\hat{Y}^j_s|+\frac{1}{N}\left|\sum_{j=1}^N(\bar{Y}^j_s-\E[\bar{Y}^j_s])\right|\right\}.
    \end{split}\end{equation}
    Then, we have
    \begin{align*}
       \E\left[\int_0^T|\hat{f}^i_s|^2ds\right]\leq C\left\{\E\left[\sup_{t\in[0,T]}|\hat{Y}^i_t|^2\right]+\frac{1}{N}\sum_{j=1}^N \E\left[\sup_{t\in[0,T]}|\hat{Y}^j_t|^2\right]+\E\left[\int_0^T\left|\frac{1}{N}\sum_{j=1}^N(\bar{Y}^j_s-\E[\bar{Y}^j_s])\right|^2ds\right]\right\}. 
    \end{align*}
      Recall that $\bar{Y}^i$, $1\leq i\leq N$, are independent copies of $Y\in \mathcal{S}^2$. It follows that there exists a constant $C$ independent of $N,T$, such that for any $t\in[0,T]$
    \begin{align}\label{var of Yj}
        \E\left[\left|\frac{1}{N}\sum_{j=1}^N(\bar{Y}^j_t-\E[\bar{Y}^j_t])\right|^2\right]=\frac{1}{N^2}\sum_{j=1}^N  \mathbb{V}[\bar{Y}^j_t]\leq \frac{C}{N},
    \end{align}
    where $\mathbb{V}[\xi]$ stands for the variance of $\xi$. All the above analysis indicates that 
    \begin{align*}
       \E\left[\sup_{t\in[0,T]}|\hat{K}_t|^2\right]\leq C\left\{\left(\E\left[\sup_{t\in[0,T]}|\hat{Y}^i_t|^2\right]\right)^{\frac{1}{2}}+\frac{1}{N}\sum_{j=1}^N \E\left[\sup_{t\in[0,T]}|\hat{Y}^j_t|^2\right]+\frac{1}{N}\right\}. 
    \end{align*}

  \textbf{Step 2.} We show that there exists a constant $C$ independent of $N$, such that 
  \begin{equation}\label{suphatYi}\begin{split}
       \E\left[\sup_{t\in[0,T]}|\hat{Y}^i_t|^2\right] \leq &C\E\left[\int_0^T\left|\frac{1}{N}\sum_{j=1}^N(\bar{Y}^j_u-\E[\bar{Y}^j_u])\right|^2du\right]
        +C\E\left[\sup_{s\in[0,T]}|L^{(N)}_s(\widetilde{U}_s)-L_s(\widetilde{U}^i_s)|^2\right].
    \end{split}\end{equation}
  Let $\bar{U}^i$ be given as in the proof of Proposition \ref{prop4.2}. For any $t\in[0,T]$ and $1\leq i\leq N$, set $\hat{U}^i_t=\bar{U}^i_t-\widetilde{U}^i_t$. Recalling Eqs. \eqref{rep Yi},  \eqref{baryi} and \eqref{hatfi}, for any $t\geq r$, we have 
    \begin{align*}
        |\hat{Y}^i_t|\leq&|\hat{U}^i_t|+\E\left[\sup_{s\in[t,T]}|L^{(N)}_s(\bar{U}_s)-L_s(\widetilde{U}^i_s)\Big|\mathcal{F}^{(N)}_t\right] \\
        \leq&L\E\left[\int_t^T|\hat{Y}^i_s|ds\Big|\mathcal{F}^{(N)}_t\right]+L\E\left[\int_t^T\frac{1}{N}\sum_{j=1}^N|\hat{Y}^j_s|ds\Big|\mathcal{F}^{(N)}_t\right]+\E\left[\sup_{s\in[r,T]}|L^{(N)}_s(\bar{U}_s)-L^{(N)}_s(\widetilde{U}_s)\Big|\mathcal{F}^{(N)}_t\right]\\
        &+L\E\left[\int_t^T\left|\frac{1}{N}\sum_{j=1}^N (\bar{Y}^j_s-\E[\bar{Y}^j_s])\right|ds\Big|\mathcal{F}^{(N)}_t\right]+\E\left[\sup_{s\in[0,T]}|L^{(N)}_s(\widetilde{U}_s)-L_s(\widetilde{U}^i_s)|\Big|\mathcal{F}^{(N)}_t\right].
    \end{align*}
    Applying Doob's inequality and H\"{o}lder's inequality yields that 
    \begin{equation}\label{e1}\begin{split}
        \E\left[\sup_{t\in[r,T]}|\hat{Y}^i_t|^2\right]\leq &20L^2 T \E\left[\int_r^T|\hat{Y}^i_s|^2ds\right]+20L^2 T\E\left[\int_r^T\left|\frac{1}{N}\sum_{j=1}^N (\bar{Y}^j_s-\E[\bar{Y}^j_s])\right|^2ds\right]\\
        &+20L^2 T\E\left[\int_t^T\frac{1}{N}\sum_{j=1}^N|\hat{Y}^j_s|^2ds\right]+20\E\left[\sup_{s\in[r,T]}|L^{(N)}_s(\bar{U}_s)-L^{(N)}_s(\widetilde{U}_s)^2\right]\\
        &+20\E\left[\sup_{s\in[0,T]}|L^{(N)}_s(\widetilde{U}_s)-L_s(\widetilde{U}^i_s)|^2\right]
    \end{split}\end{equation}
    Recalling Eqs. \eqref{LtXLtY} and \eqref{hatfi}, we have 
    \begin{align*}
         &|L^{(N)}_s(\bar{U}_s)-L^{(N)}_s(\widetilde{U}_s)|\le \frac{M}{m}\frac{1}{N}\sum_{i=1}^N|\hat{U}^i_s|\le \frac{M}{m}\frac{1}{N}\sum_{i=1}^N\E\left[\int_s^T|\hat{f}^i_r|dr\Big|\mathcal{F}^{(N)}_s\right]\\
         \leq &2\frac{M}{m}\frac{L}{N}\sum_{i=1}^N\E\left[\int_s^T|\hat{Y}^i_u|du\Big|\mathcal{F}^{(N)}_s\right]+L\frac{M}{m}\E\left[\int_s^T\frac{1}{N}\left|\sum_{j=1}^N(\bar{Y}^j_s-\E[\bar{Y}^j_s])\right|du\Big|\mathcal{F}^{(N)}_s\right].
    \end{align*}
    Applying Doob's inequality and H\"{o}lder's inequality, we obtain that 
    \begin{equation}\label{e3}\begin{split}
        \E\left[\sup_{s\in[r,T]}|L^{(N)}_s(\bar{U}_s)-L^{(N)}_s(\widetilde{U}_s)^2\right]\leq &32\frac{M^2L^2 T}{m^2}\E\left[\int_r^T\sum_{i=1}^N\frac{1}{N}|\hat{Y}^i_u|^2du\right]\\
        &+8\frac{M^2L^2T}{m^2}\E\left[\int_r^T\left|\frac{1}{N}\sum_{j=1}^N(\bar{Y}^j_u-\E[\bar{Y}^j_u])\right|^2du\right].
    \end{split}\end{equation}
    Plugging Eqs. \eqref{e3} into \eqref{e1} implies that 
    \begin{equation}\label{e4}\begin{split}
        \E\left[\sup_{t\in[r,T]}|\hat{Y}^i_t|^2\right]\leq &C \E\left[\int_r^T|\hat{Y}^i_s|^2ds\right]+C\E\left[\int_r^T\left|\frac{1}{N}\sum_{j=1}^N(\bar{Y}^j_u-\E[\bar{Y}^j_u])\right|^2du\right]\\
        &+C\E\left[\int_r^T\sum_{i=1}^N\frac{1}{N}|\hat{Y}^i_u|^2du\right]+C\E\left[\sup_{s\in[0,T]}|L^{(N)}_s(\widetilde{U}_s)-L_s(\widetilde{U}^i_s)|^2\right],
    \end{split}\end{equation}
    where $C$ is a constant depending on $L,M,m,T$. 
    Summing over $i$, we obtain 
    \begin{equation}\label{e4'}\begin{split}
        \E\left[\frac{1}{N}\sum_{i=1}^N\sup_{t\in[r,T]}|\hat{Y}^i_t|^2\right]\leq &C \int_r^T\E\left[\frac{1}{N}\sum_{i=1}^N\sup_{t\in[s,T]}|\hat{Y}^i_t|^2\right]ds
        +C\E\left[\int_0^T\left|\frac{1}{N}\sum_{j=1}^N(\bar{Y}^j_u-\E[\bar{Y}^j_u])\right|^2du\right]\\
        &+C\E\left[\sup_{s\in[0,T]}|L^{(N)}_s(\widetilde{U}_s)-L_s(\widetilde{U}^i_s)|^2\right].
    \end{split}\end{equation}
    It follows from the Gronwall inequality that 
    \begin{align*}
       \E\left[\frac{1}{N}\sum_{i=1}^N\sup_{t\in[0,T]}|\hat{Y}^i_t|^2\right] \leq &C\E\left[\int_0^T\left|\frac{1}{N}\sum_{j=1}^N(\bar{Y}^j_u-\E[\bar{Y}^j_u])\right|^2du\right]
        +C\E\left[\sup_{s\in[0,T]}|L^{(N)}_s(\widetilde{U}_s)-L_s(\widetilde{U}^i_s)|^2\right].
    \end{align*}
    Coming back to Eq. \eqref{e4} and applying Gronwall inequality again, we finally obtain Eq. \eqref{suphatYi}.

\textbf{Step 3.} It remains to prove the convergence rate for  $\E\left[\sup_{t\in[0,T]}|\hat{Y}^i_t|^2\right]$.  For any $t\in[0,T]$, denote $\mu_t$ and $\mu^{(N)}_t$ the common law and empirical law of the random variables $\{\widetilde{U}^i_t\}_{i=1}^N$, respectively, i.e., $\mu^{(N)}_t=\frac{1}{N}\sum_{i=1}^N\delta_{\widetilde{U}^i_t}$. For any $(t,x,\mu)\in [0,T]\times\mathbb{R}\times \mathcal{P}_1(\mathbb{R})$, we define
    \begin{align*}
        H(t,x,\mu):=\int h(t,x+y)\mu(dy).
    \end{align*}
    For each $(t,\mu)\in[0,T]\times\mathcal{P}_1(\mathbb{R})$, it is easy to check that $H(t,\cdot,\mu)$ is increasing and for any $x,y\in\mathbb{R}$,
    \begin{align}\label{Lip H}
        m|x-y|\leq |H(t,x,\mu)-H(t,y,\mu)|\leq M|x-y|.
    \end{align}
    Recalling Eqs. \eqref{tildeLN} and \eqref{tildeL}, we have 
    \begin{equation}\label{LNL}\begin{split}
        |L^{(N)}_t(\widetilde{U}_t)-L_t(\widetilde{U}^i_t)|&\leq |\tilde{L}^{(N)}_t(\widetilde{U}_t)-\tilde{L}_t(\widetilde{U}^i_t)|\\
        &\leq \frac{1}{m}|H(t,\tilde{L}^{(N)}_t(\widetilde{U}_t),\mu^{(N)}_t)-H(t,\tilde{L}_t(\widetilde{U}^i_t),\mu^{(N)}_t)|\\
        &= \frac{1}{m}|H(t,\tilde{L}_t(\widetilde{U}^i_t),\mu^{(N)}_t)-H(t,\tilde{L}_t(\widetilde{U}^i_t),\mu_t)|.
    \end{split}\end{equation}

    \textbf{Case 1.} Set $$\hat{H}_t=H(t,\tilde{L}_t(\widetilde{U}^i_t),\mu^{(N)}_t)-H(t,\tilde{L}_t(\widetilde{U}^i_t),\mu_t)=\frac{1}{N}\sum_{i=1}^N(h(t,V^i_t)-\E[h(t,V^i_t)]),$$
    where $V^i_t=\widetilde{U}^i_t+\tilde{L}_t(\widetilde{U}^i_t)$. By Lemma \ref{Lip}, let $\{\psi_t\}_{t\in[0,T]}$ be the derivative of $\{\tilde{L}_t(\widetilde{U}^i_t)\}_{t\in[0,T]}$, which is bounded. It follows from It\^{o}'s formula that 
    \begin{align*}
        \hat{H}_t=&\frac{1}{N}\sum_{i=1}^N(h(T,V^i_T)-\E[h(T,V^i_T)])-\int_t^T\frac{1}{N}\sum_{i=1}^N\left\{\frac{\partial h}{\partial t}(u,V^i_u)-\E\left[\frac{\partial h}{\partial t}(u,V^i_u)\right]\right\}du\\
        &-\int_t^T\frac{1}{N}\sum_{i=1}^N\left\{\frac{\partial h}{\partial x}(u,V^i_u)(\psi_u-\bar{f}^i_u)-\E\left[\frac{\partial h}{\partial x}(u,V^i_u)(\psi_u-\bar{f}^i_u)\right]\right\}du\\
        &-\frac{1}{2}\int_t^T\frac{1}{N}\sum_{i=1}^N\left\{\frac{\partial^2 h}{\partial x^2}(u,V^i_u)|\bar{Z}^i_u|^2-\E\left[\frac{\partial^2 h}{\partial x^2}(u,V^i_u)|\bar{Z}^i_u|^2\right]\right\}du-\frac{1}{N}\int_t^T\sum_{i=1}^N\frac{\partial h}{\partial x}(u,\bar{V}^i_u)\bar{Z}^i_u dB^i_u.
    \end{align*}
    Applying H\"{o}lder's inequality and Doob's inequality, there exists a constant $C$ independent of $N$, such that 
    \begin{align*}
        \E\left[\sup_{t\in[0,T]}|\hat{H}_t|^2\right]\leq &\frac{C}{N} \left\{\mathbb{V}[h(T,V^1_T)]+\int_0^T \mathbb{V}\left[\frac{\partial h}{\partial t}(u,V^1_u)\right]du+\int_0^T \mathbb{V}\left[\frac{\partial h}{\partial x}(u,V^1_u)(\psi_u-\bar{f}^1_u)\right]du\right\}\\
        &+\frac{C}{N} \left\{\int_0^T \E\left[\left|\frac{\partial h}{\partial x}(u,V^1_u)\bar{Z}^1_u\right|^2\right]du+\int_0^T \mathbb{V}\left[\frac{\partial^2 h}{\partial x^2}(u,V^1_u)|\bar{Z}^1_u|^2\right]du\right\}\\
        \leq & \frac{C}{N}\left\{1+\E[|V^1_T|]^2+\E\left[\int_0^T|\bar{f}^1_u|^2du\right]+\E\left[\int_0^T|\bar{Z}^1_u|^2 du\right]+\E\left[\int_0^T|\bar{Z}^1_u|^4 du\right]\right\}\\
        \leq & \frac{C}{N}\Bigg\{1+\E[|\xi|]^2+\E\left[\int_0^T|f(u,0,0)|^2du\right]+\E\left[\sup_{t\in[0,T]}|\bar{Y}^1_t|^2\right]\\
        &+\E\left[\int_0^T|\bar{Z}^1_u|^2 du\right]+\E\left[\int_0^T|\bar{Z}^1_u|^4 du\right]\Bigg\}\leq \frac{C}{N},
        \end{align*}
        which together with Eqs. \eqref{var of Yj}, \eqref{suphatYi} and \eqref{LNL} implies that desired result.

        \textbf{Case 2.} By Eqs. \eqref{Lip H} and \eqref{LNL}, we have 
        \begin{align}\label{supLNL}
            \E\left[\sup_{s\in[0,T]}|L^{(N)}_s(\widetilde{U}_s)-L_s(\widetilde{U}^i_s)|^2\right]\leq \frac{M^2}{m^2}\E\left[\sup_{t\in[0,T]} W^2_1(\mu^{(N)}_t,\mu_t)\right]. 
        \end{align}
        Since for any $0\le s\leq t\leq T$, we have
        \begin{align*}
            \widetilde{U}^i_s-\widetilde{U}^i_t=\int_s^t f^i(u,\bar{Y}^i_u,\E[\bar{Y}^i_u])du-\int_s^t \bar{Z}^i_udB^i_u.
            \end{align*}
            It is easy to check that, for any $2\leq q\le p$, 
     \begin{align*}
    \E[|\widetilde{U}^i_s-\widetilde{U}^i_t|^q]\leq &C\left\{\E\left[\int_s^t |f^i(u,\bar{Y}^i_u,\E[\bar{Y}^i_u])|^q du\right]|t-s|^{\frac{q}{2}}+\left(\E\left[\int_s^t |\bar{Z}^i_u|^2 du\right]\right)^{\frac{q}{2}}\right\}\\
    \leq &C\left\{\left(\E\left[\int_s^t (|f^i(u,0,0)|^q +|\bar{Y}^i_u|^q)du\right]\right)|t-s|^{\frac{q}{2}}+\sup_{t\in[0,T]} \E[|\bar{Z}^i_u|^q] |t-s|^{\frac{q}{2}}\right\}\\
    \leq & C|t-s|^{\frac{q}{2}},
    \end{align*}
    where we have used Theorem 1 in \cite{HMW}. Then, following the proof of Theorem 3.3 in \cite{BCGL}, there exists a constant $C$ independent of $N$, such that  
    \begin{align*}
        \E\left[\sup_{t\in[0,T]} W^2_1(\mu^{(N)}_t,\mu_t)\right]\leq \frac{C}{n^{\frac{1}{2}}},
    \end{align*}
    which together with Eqs. \eqref{var of Yj}, \eqref{suphatYi} and \eqref{supLNL} implies that desired result.
\end{proof}

\begin{remark}
  (1)  Consider the following mean-field BSDE with mean reflection:
\begin{displaymath}
\begin{cases}
\bar{Y}^i_t=\xi^i+\int_t^T f^i(s,\bar{Y}^i_s,\P_{\bar{Y}^i_s})ds-\int_t^T \bar{Z}^i_s dB^i_s+K_T-K_t, \\
\E[h(t,\bar{Y}^i_t)]\geq 0, t\in[0,T] \textrm{ and} \int_0^T \E[h(t,\bar{Y}^i_t)]dK_t=0.
\end{cases}
\end{displaymath}
We may use the solution to the particle system in Remark \ref{rem distribution} to approximate $(\bar{Y}^i,\bar{Z}^i,K)$. Moreover, similar convergence result as in Theorem \ref{thm4.3} still holds. 

(2) The propagation of chaos for mean-field reflected BSDE \eqref{nonlinearyz'} has been established in \cite{DDZ} (see Proposition 4.3 in \cite{DDZ}). Compared with their result, we do not need any additional requirement for the Lipschitz constant of the loss function $h$. Moreover, we may provide the explicit rate of convergence.
\end{remark}

\section{The case of linear reflection and the propagation of chaos}

In this section, we study the case of linear reflection, which means that the loss function $h$ satisfies the following assumption.
\begin{itemize}
\item[(H3')] $h(t,x)=ax+b$ for some $a>0$ and $b\in\mathbb{R}$.
\end{itemize}

\begin{remark}\label{rH2}
    Under condition (H3'), the following set 
    \begin{align*}
        \{y\in\mathbb{R}^N:\sum_{i=1}^N h(y^i)\geq 0\}
    \end{align*}
    is convex  with normal direction proportional to $(a,\cdots,a)$.
\end{remark}

When investigating the propagation of chaos for mean-field mean reflected BSDE, if the loss function $h$ is linear, we are able to deal with the case that $f$ depends on $Z$ and its distribution. Throughout this section, we assume that (H1), (H2) and (H3') hold.  We consider the following multi-dimensional reflected BSDE:
\begin{equation}\label{eq7'}
\begin{cases}
Y^i_t=\theta^i+\int_t^T f^i\left(s,Y^i_s,\frac{1}{N}\sum_{i=1}^N Y^i_s,Z^{i,i}_s,\frac{1}{N}\sum_{i=1}^N Z^{i,i}_s\right) ds\\
 \ \ \ \ \ \ \ \ \ -\int_t^T\sum_{j=1}^N Z^{i,j}_sd B^j_s+K^{(N)}_T-K^{(N)}_t, \ \forall 1\leq i\leq N, \\
\frac{1}{N}\sum_{i=1}^N h(Y^i_t)\geq 0, \ t\in[0,T] \textrm{ and }
\int_0^T\frac{1}{N}\sum_{i=1}^N h(Y^i_t)dK^{(N)}_t=0.
\end{cases}
\end{equation}

\begin{theorem}\label{thm5.1'}
 The reflected BSDE \eqref{eq7'} has a unique solution $(\{Y^i,Z^i\}_{1\leq i\leq N},K^{(N)})$ with $K^{(N)}\in \mathcal{A}^2(\mathbb{F}^{(N)})$, $Y^i\in \mathcal{S}^2(\mathbb{F}^{(N)})$ and $Z^i\in \mathcal{H}^2(\mathbb{F}^{(N)};\mathbb{R}^N)$ for $i=1,\cdots,N$. Moreover, there exists a constant $C$ independent of $N$, such that
\begin{align*}
    \E\left[|K^{(N)}_T|^2\right]\leq C.
\end{align*}
\end{theorem}

\begin{proof}
    By Remark \ref{rH2} and Theorem 5.9 in \cite{GP}, the reflected BSDE \eqref{eq7'} admits a unique solution. It remains to prove the last assertion. For this purpose, set $c=\frac{b}{a}$. Applying It\^{o}'s formula to $e^{\alpha t }(Y^i_t+c)^2$, where $\alpha$ is a positive constant to be determined later, we obtain that 
    \begin{align*}
        &e^{\alpha t}(Y^i_t+c)^2+\alpha\int_t^T e^{\alpha s }(Y^i_s+c)^2 ds+\int_t^T e^{\alpha s}|Z^i_s|^2ds\\
        =&e^{\alpha T}(Y^i_T+c)^2+2\int_t^T e^{\alpha s}(Y^i_s+c)f^i\left(s,Y^i_s,\frac{1}{N}\sum_{i=1}^N Y^i_s,Z^{i,i}_s,\frac{1}{N}\sum_{i=1}^N Z^{i,i}_s\right) ds\\
        &-2\int_t^T\sum_{j=1}^N e^{\alpha s}(Y^i_s+c)Z^{i,j}_sd B^j_s+2\int_t^T e^{\alpha s}(Y^i_s+c)dK^{(N)}_s.
    \end{align*}
    By the Lipschitz assumption for $f$, it is easy to check that 
    \begin{align*}
        &2(Y^i_s+c)f^i\left(s,Y^i_s,\frac{1}{N}\sum_{i=1}^N Y^i_s,Z^{i,i}_s,\frac{1}{N}\sum_{i=1}^N Z^{i,i}_s\right)\\
        \leq &2|Y^i_s+c|\left(|f^i(s,-c,-c,0,0)|+L|Y^i_s+c|+L\left|\frac{1}{N}\sum_{i=1}^N (Y^i_s+c)\right|+L|Z^{i,i}_s|+L\left|\frac{1}{N}\sum_{i=1}^N Z^{i,i}_s\right|\right)\\
        \leq &|f^i(s,-c,-c,0,0)|^2+c_1\frac{1}{N}\sum_{i=1}^N(Y^i_s+c)^2+c_2|Z^{i,i}_s|^2+c_3\frac{1}{N}\sum_{i=1}^N(Z^{i,i}_s)^2+\tilde{L}(Y^i_s+c)^2.
    \end{align*}
    where $c_1,c_2,c_3$ are positive constants to be determined later and
    \begin{align*}
        \tilde{L}=\left(1+2L+L^2\left(\frac{1}{c_1}+\frac{1}{c_2}+\frac{1}{c_3}\right)\right).
    \end{align*} 
  All the above analysis implies that 
    \begin{align*}
        &\E\left[e^{\alpha t}(Y^i_t+c)^2+(\alpha-\tilde{L})\int_t^T e^{\alpha s }(Y^i_s+c)^2 ds+(1-c_2)\int_t^T e^{\alpha s}|Z^i_s|^2ds\right]\\
        \leq & \E\left[e^{\alpha T}(Y^i_T+c)^2+\int_0^T e^{\alpha s}|f^i(s,-c,-c,0,0)|^2ds+c_1\int_t^T e^{\alpha s}\frac{1}{N}\sum_{i=1}^N(Y^i_s+c)^2ds \right]\\
        &+\E\left[c_3\int_t^T e^{\alpha s}\frac{1}{N}\sum_{i=1}^N(Z^{i,i}_s)^2ds+\frac{2}{a}\int_t^T e^{\alpha s}h(Y^i_s)dK^{(N)}_s\right].
    \end{align*}
    Summing over $i$ and noting that $f^i$, $1\leq i\leq N$ are IID copies of $f$, we obtain that 
    \begin{align*}
        &\frac{1}{N}\sum_{i=1}^N\E\left[e^{\alpha t}(Y^i_t+c)^2+(\alpha-\tilde{L})\int_t^T e^{\alpha s }(Y^i_s+c)^2 ds+(1-c_2)\int_t^T e^{\alpha s}|Z^i_s|^2ds\right]\\
        \leq & \E\left[\frac{1}{N}\sum_{i=1}^Ne^{\alpha T}(Y^i_T+c)^2+\int_0^T e^{\alpha s}|f(s,-c,-c,0,0)|^2ds+c_1\int_t^T e^{\alpha s}\frac{1}{N}\sum_{i=1}^N(Y^i_s+c)^2ds \right]\\
        &+\E\left[c_3\int_t^T e^{\alpha s}\frac{1}{N}\sum_{i=1}^N(Z^{i,i}_s)^2ds+\frac{2}{a}\int_t^T e^{\alpha s}\frac{1}{N}\sum_{i=1}^Nh(Y^i_s)dK^{(N)}_s\right].
    \end{align*}
    Choosing $c_1=1,c_2=c_3=\frac{1}{4}, \alpha=\tilde{L}+c_1+1$, using the Skorokhod condition and noting that 
    \begin{align*}
       Y^i_T=\xi^i+L^{(N)}_T(\xi)=\xi^i+\left(\frac{1}{N}\sum_{j=1}^N \xi^j+\frac{b}{a}\right)^-,
   \end{align*}
   where $\xi^i$, $1\leq i\leq N$, are IID copies of $\xi$,  there exists a constant $C$ independent of $N$, such that 
   \begin{equation}\begin{split}\label{yit+c}
       &\sup_{t\in[0,T]}\frac{1}{N}\sum_{i=1}^N\E\left[(Y^i_t+c)^2+\int_t^T (Y^i_s+c)^2 ds+\int_t^T |Z^i_s|^2ds\right]\\
        \leq & C\E\left[\frac{1}{N}\sum_{i=1}^N(Y^i_T+c)^2+\int_0^T |f(s,-c,-c,0,0)|^2ds\right]\\
        \leq & C\left(1+\E\left[|\xi|^2+\int_0^T |f(s,0,0,0,0)|^2ds\right]\right).
   \end{split}\end{equation}

   Finally, for each $1\leq i\leq N$, we have
   \begin{align*}
       K^{(N)}_t=Y^i_0-Y^i_T-\int_0^T f^i\left(s,Y^i_s,\frac{1}{N}\sum_{i=1}^N Y^i_s,Z^{i,i}_s,\frac{1}{N}\sum_{i=1}^N Z^{i,i}_s\right) ds+\int_0^T\sum_{j=1}^N Z^{i,j}_sd B^j_s.
   \end{align*}
   Simple calculation yields that 
   \begin{align*}
       \E\left[|K^{(N)}_T|^2\right]\leq& C\E\left[\int_0^T|f^i(s,0,0,0,0)|^2ds+\int_0^T|Z^i_s|^2ds+\int_0^T|Y^i_s|^2ds\right]\\
       &+\E\left[|Y^i_0|^2+|Y^i_T|^2\right]+\frac{1}{N}\E\left[\int_0^T\sum_{i=1}^N|Y^i_s|^2ds\right]+\frac{1}{N}\E\left[\int_0^T\sum_{i=1}^N|Z^i_s|^2ds\right].
   \end{align*}
   Taking arithmetic mean over $i$ and using \eqref{yit+c}, we obtain the desired result.
\end{proof}

Now, let $(\bar{Y}^i,\widehat{Z}^i,K)$ be the solution to the following mean-field BSDE with mean reflection:
\begin{displaymath}
\begin{cases}
\bar{Y}^i_t=\xi^i+\int_t^T f^i(s,\bar{Y}^i_s,\E[\bar{Y}^i_s],\widehat{Z}^i_s,\E[\widehat{Z}^i_s])ds-\int_t^T \widehat{Z}^i_s dB^i_s+K_T-K_t, \\
\E[h(t,\bar{Y}^i_t)]\geq 0, t\in[0,T] \textrm{ and} \int_0^T \E[h(t,\bar{Y}^i_t)]dK_t=0.
\end{cases}
\end{displaymath}
Then, $(\bar{Y}^i,\widehat{Z}^i,K)$, $1\leq i\leq N$ are independent copies of  $(Y,Z,K)$, the solution to \eqref{nonlinearyz}. For each $1\leq j\leq N$, set $\bar{Z}^{i,j}=\widehat{Z}^i I_{\{j=i\}}$. Therefore, the dynamics of the above mean-field BSDE with mean reflection can be rewritten as 
\begin{align*}
    \bar{Y}^i_t=\xi^i+\int_t^T f^i(s,\bar{Y}^i_s,\E[\bar{Y}^i_s],\bar{Z}^{i,i}_s,\E[\bar{Z}^{i,i}_s])ds-\int_t^T \sum_{j=1}^N\bar{Z}^{i,j}_s dB^j_s+K_T-K_t.
\end{align*}

\begin{remark}\label{rem representation}
    By the proof of Theorem 3.1 in \cite{BH} or the proof of Theorem \ref{thm4.1}, we have $Y^i_t=U^i_t+S_t$, where 
    \begin{align*}
        U^i_t=\E\left[\xi^i+\int_t^T f^i\left(s,Y^i_s,\frac{1}{N}\sum_{i=1}^N Y^i_s,Z^{i,i}_s,\frac{1}{N}\sum_{i=1}^N Z^{i,i}_s\right) ds\Big|\mathcal{F}^{(N)}_t\right]
    \end{align*}
    and $S$ is the Snell envelope of 
    \begin{align*}
        L^{(N)}_t(U_t):=\inf\left\{x\geq 0:\frac{1}{N}\sum_{i=1}^N h(t,U^i_t+x)\geq 0\right\}=\left(\frac{1}{N}\sum_{i=1}^N U^i_t+\frac{b}{a}\right)^-.
    \end{align*}
   
   Set 
   \begin{align*}
        \bar{U}^i_t&:=\E\left[\xi^i+\int_t^T f^i(s,\bar{Y}^i_s,\E[\bar{Y}^i_s],\bar{Z}^{i,i}_s,\E[\bar{Z}^{i,i}_s])ds\Big|\mathcal{F}^i_t\right]\\
        &=\E\left[\xi^i+\int_t^T f^i(s,\bar{Y}^i_s,\E[\bar{Y}^i_s],\bar{Z}^{i,i}_s,\E[\bar{Z}^{i,i}_s])ds\Big|\mathcal{F}^{(N)}_t\right]
    \end{align*}
    and 
   \begin{align*}
       L_t(\bar{U}^i_t):=\inf\{x\geq 0: \E[h(t,\bar{U}^i_t)]\geq 0\}=\left(\E[\bar{U}^i_t]+\frac{b}{a}\right)^-.
   \end{align*}
   By the proof of Proposition 7 in \cite{BEH}, we have $\bar{Y}^i_t=\bar{U}^i_t+R_t$, where    
    \begin{align*}
        R_t=\sup_{s\in[t,T]} L_s(\bar{U}^i_s)=\esssup_{\tau\in\mathcal{T}^N_{t,T}}\E\left[L_\tau(\bar{U}^i_\tau)|\mathcal{F}^{(N)}_t\right].
    \end{align*}
\end{remark}

Now, we state the main result in this section.
\begin{theorem}
   For any $1\leq i\leq N$, set 
\begin{align*}
\hat{Y}^i:=Y^i-\bar{Y}^i, \ \hat{K}:=K^{(N)}-K,\ \hat{Z}^i:=Z^i-\bar{Z}^i.
\end{align*}
Then, there exists a constant $C$ independent of $N$, such that
\begin{align*}
\E\left[\sup_{t\in[0,T]}|\hat{Y}^i_t|^2\right]\leq CN^{-1/2}, \ \E\left[\sup_{t\in[0,T]}|\hat{K}_t|^2\right]\leq CN^{-1/2}, \ \E\left[\int_0^T|\hat{Z}^i_t|^2dt\right]\leq CN^{-1/2}.
\end{align*} 
\end{theorem}

\begin{proof}
\textbf{Step 1.} We first show that, there exist a constant $C$ independent of $N$, such that for any $1\leq i\leq N$,
\begin{align}\label{step 1}
      \E\left[\int_0^T|\hat{Y}^i_s|ds+\int_0^T |\hat{Z}^i_s|ds\right]\leq \frac{C}{N^{\frac{1}{2}}}.  
    \end{align}
    
    Set 
    \begin{align*}
        \hat{f}^i_s=f^i\left(s,Y^i_s,\frac{1}{N}\sum_{i=1}^N Y^i_s,Z^{i,i}_s,\frac{1}{N}\sum_{i=1}^N Z^{i,i}_s\right)-f^i(s,\bar{Y}^i_s,\E[\bar{Y}^i_s],\bar{Z}^{i,i}_s,\E[\bar{Z}^{i,i}_s]).
    \end{align*}
    It is easy to check that 
    \begin{align*}
        \hat{Y}^i_t=L^{(N)}_T(U_T)+\int_t^T \hat{f}^i_s ds-\int_t^T \sum_{j=1}^N\hat{Z}^{i,j}_s dB^j_s+\hat{K}_T-\hat{K}_t.
    \end{align*}
    Applying It\^{o}'s formula to $e^{\alpha t}|\hat{Y}^i_t|^2$ and taking expectations, where $\alpha$ is a positive constant to be determined later, we obtain that
    \begin{equation}\label{eqO}\begin{split}
        &\E\left[|\hat{Y}^i_0|^2+\alpha \int_0^Te^{\alpha s}|\hat{Y}^i_s|ds+\int_0^T e^{\alpha s}|\hat{Z}^i_s|ds\right]\\
        =&\E\left[e^{\alpha T}|L^{(N)}_T(U_T)|^2+2\int_0^T e^{\alpha s}\hat{Y}^i_s\hat{f}^i_s ds+2\int_0^T e^{\alpha s}\hat{Y}^i_s d\hat{K}_s\right]\\
        =&\E\left[e^{\alpha T}|L^{(N)}_T(U_T)|^2+2\int_0^T e^{\alpha s}\hat{Y}^i_s\hat{f}^i_s ds+\frac{2}{a}\int_0^T e^{\alpha s}(h(s,Y^i_s)-h(s,\bar{Y}^i_s)) d\hat{K}_s\right].
    \end{split}\end{equation}
    Simple calculation yields that 
    \begin{align*}
        2\hat{Y}^i_s\hat{f}^i_s\leq & 2L|\hat{Y}^i_s|\left(|\hat{Y}^i_s|+|\hat{Z}^{i,i}_s|+\left|\frac{1}{N}\sum_{j=1}^N Y^j_s-\E[\bar{Y}^i_s]\right|+\left|\frac{1}{N}\sum_{j=1}^N Z^{j,j}_s-\E[\bar{Z}^{i,i}_s]\right|\right)\\
        \leq &\left(2L+L^2\left(\frac{1}{c_1}+\frac{1}{c_2}+\frac{1}{c_3}\right)\right)|\hat{Y}^i_s|^2+c_1|\hat{Z}^{i,i}_s|^2\\
        &+c_2\left|\frac{1}{N}\sum_{j=1}^N Y^j_s-\E[\bar{Y}^i_s]\right|^2+c_3\left|\frac{1}{N}\sum_{j=1}^N Z^{j,j}_s-\E[\bar{Z}^{i,i}_s]\right|^2.
    \end{align*}
    Since $\bar{Y}^i$, $1\leq i\leq N$, are independent copies of $Y\in \mathcal{S}^2$, it is easy to check that 
\begin{align*}
    \frac{1}{N}\sum_{j=1}^N Y^j_s-\E[\bar{Y}^i_s]=\frac{1}{N}\sum_{j=1}^N \hat{Y}^j_s+\frac{1}{N}\sum_{j=1}^N (\bar{Y}^j_s-\E[\bar{Y}^j_s])
\end{align*}
    Consequently, we have
    \begin{equation*}
\begin{split}
        \E\left[\int_0^Te^{\alpha s}\left|\frac{1}{N}\sum_{j=1}^N {Y}^j_s-\E[\bar{Y}^i_s]\right|^2ds\right]\leq &2\E\left[\int_0^T e^{\alpha s}\sum_{i=1}^N\frac{1}{N}|\hat{Y}^i_s|^2ds\right]\\
        &+2\E\left[\int_0^T e^{\alpha s}\left|\frac{1}{N}\sum_{i=1}^N(\bar{Y}^i_s-\E[\bar{Y}^i_s])\right|^2ds\right].
    \end{split}\end{equation*}
    Similarly, we have
    \begin{equation*}
\begin{split}
        \E\left[\int_0^Te^{\alpha s}\left|\frac{1}{N}\sum_{j=1}^N {Z}^{j,j}_s-\E[\bar{Z}^{i,i}_s]\right|^2ds\right]\leq &2\E\left[\int_0^T e^{\alpha s}\sum_{i=1}^N\frac{1}{N}|\hat{Z}^{i,i}_s|^2ds\right]\\
        &+2\E\left[\int_0^T e^{\alpha s}\left|\frac{1}{N}\sum_{i=1}^N(\bar{Z}^{i,i}_s-\E[\bar{Z}^{i,i}_s])\right|^2ds\right].
    \end{split}\end{equation*}
    Therefore, we obtain that 
    \begin{equation}\begin{split}\label{eqI}
        \E\left[2\int_0^T e^{\alpha s}\hat{Y}^i_s\hat{f}^i_s ds\right]\leq & \left(2L+L^2\left(\frac{1}{c_1}+\frac{1}{c_2}+\frac{1}{c_3}\right)\right)\E\left[\int_0^Te^{\alpha s}|\hat{Y}^i_s|ds\right]+c_1\E\left[\int_0^T e^{\alpha s}|\hat{Z}^i_s|ds\right]\\
        &+2c_2\E\left[\int_0^T e^{\alpha s}\sum_{i=1}^N\frac{1}{N}|\hat{Y}^i_s|^2ds\right]
        +2c_2\E\left[\int_0^T e^{\alpha s}\left|\frac{1}{N}\sum_{i=1}^N(\bar{Y}^i_s-\E[\bar{Y}^i_s])\right|^2ds\right]\\
        &+2c_3\E\left[\int_0^T e^{\alpha s}\sum_{i=1}^N\frac{1}{N}|\hat{Z}^{i}_s|^2ds\right]+2c_3\E\left[\int_0^T e^{\alpha s}\left|\frac{1}{N}\sum_{i=1}^N(\bar{Z}^{i,i}_s-\E[\bar{Z}^{i,i}_s])\right|^2ds\right].
    \end{split}\end{equation}
   
    On the other hand, since $\frac{1}{N}\sum_{i=1}^Nh(s,Y^i_s)\geq 0$ and $K$ is nondecreasing, we have 
    \begin{align*}
        \E\left[\int_0^T e^{\alpha s}\frac{1}{N}\sum_{i=1}^Nh(s,Y^i_s) dK_s\right]\geq 0.
    \end{align*}
    Recalling that $\bar{Y}^i$, $1\leq i\leq N$, are independent copies of $Y$, we obtain that 
    \begin{align*}
        \E\left[\int_0^T e^{\alpha s}\frac{1}{N}\sum_{i=1}^Nh(s,\bar{Y}^i_s) dK_s\right]=\int_0^T e^{\alpha s}\E[h(s,{Y}_s)] dK_s=0.
    \end{align*}
    Noting that $\E[h(s,\bar{Y}^i_s)]\geq 0$, $K^{(N)}$ is nondecreasing and $\bar{Y}^i_t=\bar{U}^i_t+K_T-K_t$, we have
    \begin{align*}
        -\E\left[\int_0^T e^{\alpha s}\frac{1}{N}\sum_{i=1}^Nh(s,\bar{Y}^i_s) dK^{(N)}_s\right]\leq &\E\left[\int_0^T e^{\alpha s}\frac{1}{N}\sum_{i=1}^N (\E[h(s,\bar{Y}^i_s)]-h(s,\bar{Y}^i_s)) dK^{(N)}_s\right]\\
        =&\E\left[\int_0^T e^{\alpha s}\frac{1}{N}\sum_{i=1}^N (\E[\bar{U}^i_s]-\bar{U}^i_s) dK^{(N)}_s\right]\\
        \leq &e^{\alpha T}\left(\E\left[|K^{(N)}_T|^2\right]\right)^{\frac{1}{2}}\left(\E\left[\sup_{t\in[0,T]}\left|\frac{1}{N}\sum_{i=1}^N (\E[\bar{U}^i_s]-\bar{U}^i_t)\right|^2\right]\right)^{\frac{1}{2}}.
    \end{align*}
    Recalling the Skorokhod condition in \eqref{eq7'}, all the above analysis indicates that  
    \begin{equation}\label{eqII}\begin{split}
        &\E\left[\int_0^T e^{\alpha s}\frac{1}{N}\sum_{i=1}^N(h(s,Y^i_s)-h(s,\bar{Y}^i_s)) d\hat{K}_s\right]\\
        =&\E\left[\int_0^T e^{\alpha s}\frac{1}{N}\sum_{i=1}^Nh(s,Y^i_s) dK^{(N)}_s\right]-\E\left[\int_0^T e^{\alpha s}\frac{1}{N}\sum_{i=1}^Nh(s,Y^i_s) dK_s\right]\\
        &-\E\left[\int_0^T e^{\alpha s}\frac{1}{N}\sum_{i=1}^Nh(s,\bar{Y}^i_s) dK^{(N)}_s\right]+\E\left[\int_0^T e^{\alpha s}\frac{1}{N}\sum_{i=1}^Nh(s,\bar{Y}^i_s) dK_s\right]\\
        \leq &e^{\alpha T}\left(\E\left[|K^{(N)}_T|^2\right]\right)^{\frac{1}{2}}\left(\E\left[\sup_{t\in[0,T]}\left|\frac{1}{N}\sum_{i=1}^N (\E[\bar{U}^i_s]-\bar{U}^i_t)\right|^2\right]\right)^{\frac{1}{2}}.
    \end{split}\end{equation}

    Now, we choose 
    \begin{align*}
        c_1=\frac{1}{4}, \ c_2=1, \ c_3=\frac{1}{8}, \ \alpha=\left(2L+L^2\left(\frac{1}{c_1}+\frac{1}{c_2}+\frac{1}{c_3}\right)\right)+2c_2+\frac{1}{2}.
    \end{align*}
    Pluggging Eq. \eqref{eqI} into \eqref{eqO} and then summing over $i$, together with Eq. \eqref{eqII} and Theorem \ref{thm5.1'}, there exists a constant $C$ independent of $N$, such that 
    \begin{align*}
        &\frac{1}{N}\sum_{i=1}^N\E\left[\int_0^T|\hat{Y}^i_s|ds+\int_0^T |\hat{Z}^i_s|ds\right]\\
        \leq & C\left\{\E\left[|L^{(N)}_T(U_T)|^2\right]+\left(\E\left[\sup_{t\in[0,T]}\left|\frac{1}{N}\sum_{i=1}^N (\E[\bar{U}^i_t]-\bar{U}^i_t)\right|^2\right]\right)^{\frac{1}{2}}\right\}\\
        &+C\left\{\E\left[\int_0^T \left|\frac{1}{N}\sum_{i=1}^N(\bar{Y}^i_s-\E[\bar{Y}^i_s])\right|^2ds\right]+\E\left[\int_0^T \left|\frac{1}{N}\sum_{i=1}^N(\bar{Z}^{i,i}_s-\E[\bar{Z}^{i,i}_s])\right|^2ds\right]\right\}.
    \end{align*}
    
    Applying Eq. \eqref{var of Yj}, we have 
    \begin{align}\label{eqIII}
        \E\left[\int_0^T \left|\frac{1}{N}\sum_{i=1}^N(\bar{Y}^i_s-\E[\bar{Y}^i_s])\right|^2ds\right]\leq \frac{C}{N}.
    \end{align}
    Recalling that $\bar{Z}^{i,i}$, $1\leq i\leq N$, are independent copies of $Z$, it follows that 
    \begin{equation}\begin{split}\label{eqIV}
        \E\left[\int_0^T \left|\frac{1}{N}\sum_{i=1}^N(\bar{Z}^{i,i}_s-\E[\bar{Z}^{i,i}_s])\right|^2ds\right]&=\int_0^T \frac{1}{N^2}\sum_{i=1}^N \E\left[|\bar{Z}^{i,i}_s-\E[\bar{Z}^{i,i}_s]|^2\right]ds\\
        &\leq \frac{4}{N}\int_0^T \E[|Z_s|^2]ds\leq \frac{C}{N}.
    \end{split}\end{equation}
    It is easy to check that 
    \begin{align*}
       \E\left[|L^{(N)}_T(U_T)|^2\right]=\E\left[\left|\frac{1}{N}\left(\sum_{i=1}^N (\xi^i+\frac{b}{a})\right)^-\right|^2\right] \leq \frac{C}{N}. 
    \end{align*}
    Note the fact that 
    \begin{align*}
        \frac{1}{N}\sum_{i=1}^N (\bar{U}^i_t-\E[\bar{U}^i_t])=\frac{1}{N}\sum_{i=1}^N (\bar{U}^i_T-\E[\bar{U}^i_T])+\int_t^T \frac{1}{N}\sum_{i=1}^N (\bar{f}^i_s-\E[\bar{f}^i_s])ds-\frac{1}{N}\int_t^T \sum_{i=1}^N \bar{Z}^{i,i}_sdB^i_s,
    \end{align*}
    where 
    \begin{align*}
        \bar{f}^i_s=f^i(s,\bar{Y}^i_s,\E[\bar{Y}^i_s],\bar{Z}^{i,i}_s,\E[\bar{Z}^{i,i}_s]).
    \end{align*}
    It follows that 
    \begin{equation}\label{var of barUi}\begin{split}
        &\E\left[\sup_{t\in[0,T]}\left|\frac{1}{N}\sum_{i=1}^N (\E[\bar{U}^i_t]-\bar{U}^i_t)\right|^2\right]\\
      \leq &\frac{C}{N}\left\{\mathbb{V}[\bar{U}^1_T]+\int_0^T \mathbb{V}[\bar{f}^1_s]ds+\E\left[\int_0^T |\bar{Z}^1_s|^2ds\right]\right\}\\
      \leq &\frac{C}{N}\left\{\E[|\xi|^2]+\E\left[\int_0^T |f(s,Y_s,\E[Y_s],Z_s,\E[Z_s])|^2ds\right]+\E\left[\int_0^T |{Z}_s|^2ds\right]\right\}.
    \end{split}\end{equation}
    All the above analysis indicates that 
    \begin{align}\label{eq0'}
      \frac{1}{N}\sum_{i=1}^N\E\left[\int_0^T|\hat{Y}^i_s|ds+\int_0^T |\hat{Z}^i_s|ds\right]\leq \frac{C}{N^{\frac{1}{2}}}.  
    \end{align}
    As claimed in the proof of Theorem 5.3 in \cite{BH}, the law of $(\hat{Y}^i,\hat{Z}^i)$ is independent of $i$, we obtain Eq. \eqref{step 1}.

    \textbf{Step 2.} Now, we are in a position to prove the final result. Set $\hat{U}^i_t=U^i_t-\bar{U}^i_t$. By Remark \ref{rem representation} and Eq. \eqref{LtXLtY}, we have 
    \begin{align*}
        |\hat{Y}^i_t|\leq& |\hat{U}^i_t|+\E\left[\sup_{s\in[t,T]}|L^{(N)}_s(U_s)-L_s(\bar{U}^i_s)|\bigg|\mathcal{F}^{(N)}_t\right]\\
        \leq &|\hat{U}^i_t|+\E\left[\sup_{s\in[0,T]}|L^{(N)}_s({U}_s)-L^{(N)}_s(\bar{U}_s)\Big|\mathcal{F}^{(N)}_t\right]+\E\left[\sup_{s\in[0,T]}|L^{(N)}_s(\bar{U}_s)-L_s(\bar{U}^i_s)|\Big|\mathcal{F}^{(N)}_t\right]\\
         \leq &|\hat{U}^i_t|+\frac{M}{m}\E\left[\sup_{s\in[0,T]}\frac{1}{N}\sum_{i=1}^N|\hat{U}^i_s|\Big|\mathcal{F}^{(N)}_t\right]+\E\left[\sup_{s\in[0,T]}|L^{(N)}_s(\bar{U}_s)-L_s(\bar{U}^i_s)|\Big|\mathcal{F}^{(N)}_t\right].
    \end{align*}
    It follows that 
    \begin{align}\label{eq0}
        \E\left[\sup_{t\in[0,T]}|\hat{Y}^i_t|^2\right]\leq C\left\{\E\left[\sup_{t\in[0,T]}|\hat{U}^i_t|^2\right]+\frac{1}{N}\sum_{i=1}^N \E\left[\sup_{t\in[0,T]}|\hat{U}^i_t|^2\right]+\E\left[\sup_{s\in[0,T]}|L^{(N)}_s(\bar{U}_s)-L_s(\bar{U}^i_s)|^2\right]\right\}.
    \end{align}
    Recalling the definition of $L^{(N)}_t(\bar{U}_t)$ and $L_t(\bar{U}^i_t)$ in Remark \ref{rem representation}, we obtain that 
    \begin{align*}
        |L^{(N)}_t(\bar{U}_t)-L_t(\bar{U}^i_t)|=\left|\frac{1}{N}\left(\sum_{i=1}^N\bar{U}^i_t+\frac{b}{a}\right)^--\frac{1}{N}\left(\sum_{i=1}^N\E[\bar{U}^i_t]+\frac{b}{a}\right)^-\right|\leq \frac{1}{N}\left|\sum_{i=1}^N (\bar{U}^i_t-\E[\bar{U}^i_t])\right|,
    \end{align*}
    where we have used the fact that $\bar{U}^i_t$, $1\leq i\leq N$, are IID random variables. Applying Eq. \eqref{var of barUi} implies that 
    \begin{align}\label{eqV}
        \E\left[\sup_{s\in[0,T]}|L^{(N)}_s(\bar{U}_s)-L_s(\bar{U}^i_s)|^2\right]\leq \frac{C}{N}.
    \end{align}
    For each $j=1,\cdots,N$, applying Doob's inequality and H\"{o}lder's inequality, we obtain that 
    \begin{align*}
        &\E\left[\sup_{t\in[0,T]}|\hat{U}^j_t|^2\right]\leq C\E\left[\int_0^T |\hat{f}^j_t|^2dt\right]\\
        \leq &C\left\{\E\left[\int_0^T(|\hat{Y}^j_t|^2+|\hat{Z}^j_t|^2)dt\right]+\E\left[\int_0^T\left|\frac{1}{N}\sum_{i=1}^N Y^i_t-\E[\bar{Y}^j_t]\right|^2+\left|\frac{1}{N}\sum_{i=1}^N Z^{i,i}_t-\E[\bar{Z}^{j,j}_t]\right|^2dt\right]\right\}\\
        \leq &C\left\{\E\left[\int_0^T(|\hat{Y}^j_t|^2+|\hat{Z}^j_t|^2)dt\right]+\frac{1}{N}\E\left[\int_0^T\sum_{i=1}^N |\hat{Y}^i_t|^2dt\right]+\frac{1}{N}\E\left[\int_0^T\sum_{i=1}^N |\hat{Z}^{i,i}_t|^2dt\right]\right\}\\
         &+C\left\{\E\left[\int_0^T\left|\frac{1}{N}\sum_{i=1}^N (\bar{Y}^i_t-\E[\bar{Y}^i_t])\right|^2dt\right]+\E\left[\int_0^T\left|\frac{1}{N}\sum_{i=1}^N (\bar{Z}^{i,i}_t-\E[\bar{Z}^{i,i}_t])\right|^2dt\right]\right\},
    \end{align*}
    which, combining with Eqs. \eqref{step 1}, \eqref{eqIII}, \eqref{eqIV} and \eqref{eq0'}, implies that 
    \begin{align}\label{eqVI}
        \E\left[\sup_{t\in[0,T]}|\hat{U}^j_t|^2\right]\leq C\E\left[\int_0^T |\hat{f}^j_t|^2dt\right]\leq \frac{C}{N^{\frac{1}{2}}}. 
    \end{align}
    Plugging Eqs. \eqref{eqV} and \eqref{eqVI} into \eqref{eq0} yields that
    \begin{align}\label{eqVII}
        \E\left[\sup_{t\in[0,T]}|\hat{Y}^i_t|^2\right]\leq \frac{C}{N^{\frac{1}{2}}}. 
    \end{align}

    Finally, since
    \begin{align*}
        \hat{K}_t=\hat{Y}^i_0-\hat{Y}^i_t-\int_0^t \hat{f}^i_sds+\int_0^t \sum_{j=1}^N\hat{Z}^{i,j}_sdB^j_s,
    \end{align*}
    Applying Eqs. \eqref{step 1}, \eqref{eqVI} and \eqref{eqVII}, we obtain that 
    \begin{align*}
        \E\left[\sup_{t\in[0,T]}|\hat{K}_t|^2\right]\leq C\left\{\E\left[\sup_{t\in[0,T]}|\hat{Y}^i_t|^2\right]+\E\left[\int_0^T|\hat{Z}^i_s|^2ds\right]+\E\left[\int_0^T|\hat{f}^i_t|^2dt\right]\right\}\leq \frac{C}{N^{\frac{1}{2}}}. 
    \end{align*}
    The proof is complete.
\end{proof}

\begin{remark}
   For the mean-field reflected BSDE \eqref{nonlinearyz'}, suppose the barrier $l$ takes the following form
    \begin{align}
        l(t,y,\mu)=y-a\int x\mu(dx)-b.
    \end{align}
    Then, the requirement $Y_t\geq l(t,Y_t,\P_{Y_t})$ turns into our constraint $a\E[Y_t]+b\geq 0$. However, since the Lipschitz constant $\gamma_1$ for $l$ is $1$ and the generator $f$ depends on the expectation of $Z$, the convergence results and propagation of chaos in \cite{DDZ} (see Proposition 4.2, Theorem 4.2, Proposition 4.3 and Theorem 4.3 in \cite{DDZ}) cannot be applied to our case.  
\end{remark}


\begin{thebibliography}{99}

\bibitem{BEH} Briand, P., Elie, R. and Hu, Y. (2018) BSDEs with mean reflection.  Ann.  Appl. Probab., 28: 482-510.

\bibitem{BH} Briand, P. and Hibon, H. (2021) Particle systems for mean reflected BSDEs. Stochastic Processes and their Applications, 131: 253-275.

\bibitem{BCGL} Briand, P., Chaudru de Raynal, P.E., Guillin, A. and Labart, C. (2020) Particle systems and numerical schemes for mean reflected stochastic differential equations. Ann. Appl. Probab., 30(4): 1884-1909. 

\bibitem{BDLP} Buckdahn, R.,  Djehiche, B.,  Li, J. and Peng, S. (2009) Mean-field backward stochastic differential equations: a limit approach, Ann. Probab., 37(4): 1524–1565.

\bibitem{BLP} Buckdahn, R.,  Li, J. and Peng, S. (2009) Mean-field backward stochastic differential equations and related partial differential equations. Stochastic Processes and their Applications, 119: 3133-3154.


\bibitem{CHM} Chen, Y., Hamad\`{e}ne, S. and Mu, T. (2022) Mean-field doubly reflected backward stochastic differential equations. Numerical Algebra, Control and Optimization, doi:10.3934/naco.2022012.

\bibitem{CM} Cr\'{e}pey, S. and Matoussi, A. (2008) Reflected and doubly reflected BSDEs with jumps. Ann. Probab.,  18(5): 2041-2069.

\bibitem{CZ} Cui, F. and Zhao, W. (2025) Mean-field BSDEs with non-Lipschitz coefficients. J. Math. Anal. Appl., 545: 129256.

\bibitem{CK} Cvitanic, J. and Karatzas, I. (1996) Backward stochastic differential equations with reflection and Dynkin games. Ann. Probab., 24(4): 2024-2056.



\bibitem{DEH} Djehiche, B., Elie, R. and Hamad\`{e}ne, S. (2023) Mean-field reflected backward stochastic differential equations. Ann. Appl. Probab., 33(4): 2493-2518. 

\bibitem{DD} Djehiche, B. and Dumitrescu, R. (2025) Zero-sum mean-field Dynkin games: characterization and convergence. Mathematics of Operations Research, available online.

\bibitem{DDZ} Djehiche, B., Dumitrescu, R. and Zeng, J. (2025) A propagation of chaos result for weakly interacting nonlinear Snell
envelopes. Stochastic Processes and their Applications, 188: 104669.


\bibitem{EKPPQ} El Karoui, N., Kapoudjian, C., Pardoux, E., Peng, S. and Quenez, M.C. (1997) {Reflected solutions of backward SDE's, and related obstacle problems for PDE's}. Ann. Probab.,  23(2): 702-737.





\bibitem{FS} Falkowski, A. and Slomi\'{n}ski, L. (2022) Backward stochastic differential equations with mean reflection and two constraints. Bulletin des Sciences Math\'{e}matiques, 176: 103117.


\bibitem{GP} G\'{e}gout-Petit, A. and Pardoux, E. (1996) \'{E}quations diff\'{e}rentielles stochastiques  r\'{e}trogrades r\'{e}fl\'{e}chies dans un convexe. Stoch. Stoch. Rep., 57(1-2): 111-128.

\bibitem{GIOOQ} Grigorova, M., Imkeller, P., Offen, E., Ouknine, Y. and Quenez, M.C. (2017) Reflected BSDEs when the obstacle is not right-continuous and optimal stopping. Ann. Appl. Probab., 27: 172-196.





\bibitem{HHLLW} Hibon, H., Hu, Y., Lin, Y., Luo, P. and Wang, F. (2018) Quadratic BSDEs with mean reflection. Mathematical Control and Related Fields. 8: 721-738.

\bibitem{HMW'} Hu, Y., Moreau, R. and Wang, F. (2022) Quadratic mean-field reflected BSDEs. Probability, Uncertainty and Quantitative Risk, 7(3): 169-194.

\bibitem{HMW} Hu, Y., Moreau, R. and Wang, F. (2024) General mean reflected backward stochastic differential equations. Journal of Theoretical Probability, 37: 877-904.

\bibitem{HT} Hu, Y., and Tang, S. (2010) Multi-dimensional BSDE with oblique reflection and optimal switching. Probab. Theory Related Fields, 147(1-2): 89-121.

\bibitem{K1} Klimsiak, T. (2012) Reflected BSDEs with monotone generator.  Electron. J. Probab., 17(107): 1-25.



\bibitem{KLQT} Kobylanski, M., Lepeltier, J.P., Quenez, M.C. and Torres, S. (2002) Reflected BSDE with superlinear quadratic coefficient. Probability and Mathematical Statistics, 22(1): 51-83.


\bibitem{Li24} Li, H. (2024) Backward stochastic differential equations with double mean reflections. Stoch. Proc. Appl., 173: 104371.

\bibitem{LN} Li, H. and Ning, N. (2024) Propagation of chaos for doubly mean reflected BSDEs, arXiv: 2401.16617.

\bibitem{LS} Li, H. and Shi, J. (2025) Mean-field backward stochastic differential equations with double mean reflections, arXiv: 2501.10939.

\bibitem{LS'} Li, H. and Shi, J. (2026) Mean-field backward stochastic differential equations with nonlinear resistance and double mean reflections, arXiv: 2605.15781.

\bibitem{LX1} Lin, Y. and Xu, K. (2024) Mean-field reflected BSDEs driven by a marked point process, arXiv: 2401.07723.

\bibitem{LX2} Lin, Y. and Xu, K. (2025) Propagation of chaos for mean-field reflected BSDEs with jumps. Statistics and Probability Letters, 221: 110382.

\bibitem{NQW} Niu, Y., Qu, B. and Wang, F. (2025) $L^p$-solutions of multi-dimensional BSDEs with mean reflection. Stoch. Proc. Appl., 187: 104663.

\bibitem{QW} Qu, B. and Wang, F. (2023) Multi-dimensional BSDEs with mean reflection. Electron. J. Probab., 28: 1-26.

\bibitem{PX} Peng, S. and Xu, M. (2005) The smallest $g$-supermartingale and reflected BSDE with single and double $L^2$ obstacles. Ann. I. H. Poincare-PR,  41: 605-630.

\bibitem{PR} Possamaï, D. and Rodrigues, M. (2024) Reflections on BSDEs. Electron. J. Probab., 29: 1-82.


\end{thebibliography}
 \end{document}